\documentclass{article}
\usepackage[utf8]{inputenc}
\usepackage{fullpage}
\usepackage{amsmath}
\usepackage{amsfonts}
\usepackage{amssymb}
\usepackage{tikz-cd}
\usepackage{color,soul}
\usepackage{array}
\usepackage{zref-savepos}
\usepackage{amsthm}
\usepackage{multirow, bigstrut}
\usepackage[font={small,it}]{caption}
\usepackage{tikz}
\usepackage{bm}
\usetikzlibrary{arrows}

\usepackage{hyperref}
\hypersetup{
    colorlinks=true,
    linkcolor=blue,
    filecolor=magenta,  
    urlcolor=cyan,
}

\usetikzlibrary{calc}
\usepackage{comment}
\usepackage{enumerate}
\usepackage{tikz}
\usetikzlibrary{shapes.geometric}

\newtheorem{theorem}{Theorem}[section]
\newtheorem*{theorem*}{Theorem}
\newtheorem{lemma}[theorem]{Lemma}
\newtheorem{proposition}[theorem]{Proposition}
\newtheorem{corollary}[theorem]{Corollary}

\theoremstyle{definition}
\newtheorem{definition}[theorem]{Definition}
\newtheorem{remark}[theorem]{Remark}
\newtheorem{example}[theorem]{Example}

\theoremstyle{plain}

\newcommand{\C}{\mathbb{C}}
\newcommand{\R}{\mathbb{R}}
\newcommand{\N}{\mathbb{N}}
\newcommand{\Z}{\mathbb{Z}}

\DeclareMathOperator{\cpc}{Cap}
\DeclareMathOperator{\Newt}{Newt}
\DeclareMathOperator{\supp}{supp}

\DeclareMathOperator{\per}{per}

\DeclareMathOperator{\Mat}{Mat}
\DeclareMathOperator{\Prod}{Prod}
\DeclareMathOperator{\HStab}{HStab}
\DeclareMathOperator{\SLC}{SLC}
\DeclareMathOperator{\LC}{LC}
\DeclareMathOperator{\NHMat}{NHMat}
\DeclareMathOperator{\NHProd}{NHProd}
\DeclareMathOperator{\NHStab}{NHStab}

\title{Capacity Lower Bounds via Productization}
\author{Leonid Gurvits and Jonathan Leake}

\begin{document}

\maketitle

\begin{abstract}
    We give a sharp lower bound on the capacity of a homogeneous real stable polynomial, depending only on the value of its gradient at $x = \bm{1}$.
    This result implies a sharp improvement to a similar inequality proved by Linial-Samorodnitsky-Wigderson in 2000 \cite{linial2000deterministic}, which was crucial to the analysis of their permanent approximation algorithm.
    Such inequalities have played an important role in the recent work on operator scaling and its generalizations and applications \cite{garg2018algorithmic}, and in fact we use our bound to construct a new scaling algorithm for real stable polynomials.
    %
    %
    
    In addition, we give a strong improvement on previous lower bounds of the capacity of a non-homogeneous real stable polynomial, depending only on the value of its gradient at $x = \bm{1}$.
    Crucially, this new bound is independent of the degree of the polynomial, and has singly exponential dependence on the number of variables.
    This compares favorably to the bounds used recently in the fantastic work of Karlin-Klein-Oveis Gharan to give an improved approximation factor for metric TSP \cite{karlin2020slightly}, where this dependence is doubly exponential.
    Such bounds were conjectured to exist by the authors of \cite{karlin2020slightly}, and thus our new bound should imply further improvement to the approximation factor for metric TSP.
    
    The new technique we develop to prove this bound is \emph{productization}, which says that any real stable polynomial can be approximated at any point in the positive orthant by a product of linear forms.
    Beyond the results of this paper, our main hope is that this new technique will allow us to avoid ``frightening technicalities'', in the words of Laurent and Schrijver, that often accompany combinatorial lower bounds.
    %
\end{abstract}

\newpage

\tableofcontents

\newpage

\section{Introduction}

In \cite{linial2000deterministic}, Linial, Samorodnitsky, and Wigderson gave the first deterministic algorithm for approximating the permanent of a matrix with non-negative entries within an exponential factor.
Their work has been wildly influential within TCS, spawning a line of work which has even recently continued to prove fruitful, for example approximating the mixed discriminant and mixed volume \cite{gurvits2000deterministic}, non-commutative polynomial identity testing \cite{garg2019operator,garg2016deterministic,allen2018operator}, computing Brascamp-Lieb constants \cite{garg2018algorithmic}, tensor scaling problems \cite{burgisser2018efficient,franks2018operator,burgisser2019towards}, and null-cone problems \cite{burgisser2017alternating}.
At the heart of their work is a complexity analysis of the classical Sinkhorn scaling algorithm \cite{sinkhorn1964relationship}, which is an iterative process for transforming a matrix to be doubly stochastic (all row and column sums equal to 1).
Once the matrix is doubly stochastic, the classical Van der Waerden inequality for the permanent of a doubly stochastic matrix gives the desired exponential approximation.
The crux of their analysis relies on a bound on the stability of the Sinkhorn scaling procedure: if a matrix is only \emph{close} to being doubly stochastic, they give a bound on how far off the Van der Waerden inequality can be.
This is crucial, due to the iterative nature of the algorithm.

The original purpose of this paper was then to generalize this bound to the realm of stable polynomials, and perhaps even beyond to log-concave/Lorentzian polynomials \cite{anari2018log,branden2019lorentzian}.
Since the introduction of polynomial capacity in \cite{gurvits2006hyperbolic}, many similar combinatorial bounds and statements have been generalized in this way.
These bounds often give rise to exponential approximations for quantities which are \#P-hard to compute exactly, see \cite{anari2017generalization,straszak2017real,anari2018log}.
This generalization to stable polynomials has not yet been done for the Sinkhorn scaling bound of Linial-Samorodnitsky-Wigderson, and this is one of our main results.
What was surprising to us is that in the process of generalizing this bound, we were also able to strengthen the bound, and in fact our bound is tight even in the original context of Sinkhorn scaling.

What led to this strengthening was the investigation of a related question.
For a given non-negative vector $\alpha \in \R^n$, let $\Mat_n(\alpha)$ be the set of all $n \times n$ matrices $A$ with non-negative entries such that the row sums of $A$ are all 1 and the column sums are given by $\alpha_1,\ldots,\alpha_n$.
Now the question:
\[
    \text{For which $\alpha$ is it the case that \emph{all} matrices in $\Mat_n(\alpha)$ have positive permanent?}
\]
The Van der Waerden inequality implies the all-ones vector is one such ``good'' value of $\alpha$.
What is remarkable is that the set of all such good vectors can be explicitly described as the set of all $\alpha$ for which $\|\alpha - 1\|_1 < 2$ (Theorem \ref{thm:main-marginals}).
A analytic version of this fact then precisely gives our strengthening of the Linial-Samorodnitsky-Wigderson bound (Theorem \ref{thm:main}).

Beyond this strengthening of the bound, our main result is extending it to the realm of stable polynomials, and the key idea here is the notion of \emph{productization}.
The typical way the matrix permanent problem is related to stable polynomials is via the map $A \mapsto f_A(x) := \prod_{i=1}^n (Ax)_i$ which maps an $n \times n$ matrix with non-negative entries to a product of linear forms in $n$ variables.
The permanent of $A$ is precisely the coefficient of $x_1x_2 \cdots x_n$ in $f_A$, and thus the permanent approximation problem is a specific case of the more general stable polynomial coefficient approximation.

Our main idea of productization is then a local converse to this fact.
Morally speaking, it says that any homogeneous real stable polynomial can be locally approximated by such a product of linear forms at any specific point in the positive orthant (Theorem \ref{thm:main-productization}).
That is, for any such $p$ and $x$, there is an $A$ such that
\[
    p(x) = \prod_{i=1}^n (Ax)_i \qquad \text{and} \qquad \nabla p(\mathbf{1}) = \nabla \left[\prod_{i=1}^n (Ax)_i\right](\mathbf{1}).
\]
This fact immediately implies the extension of the strengthened Linial-Samorodnitsky-Wigderson bound to stable polynomials as a corollary (see Corollary \ref{cor:hstab-cap-bound}).
The power of this technique is that is allows one to use combinatorial arguments on non-negative matrices (products of linear forms) to prove more general statements about stable polynomials.
And further, this technique may extend to log-concave polynomials in general as well.
Because of this, we believe that this productization technique has the potential to drive new connections between TCS and the various log-concave polynomial classes.

As a specific example, our main bound for non-homogeneous polynomials has a strong resemblance to the bounds utilized in the recent exciting work on metric TSP by Karlin, Klien, and Oveis Gharan \cite{karlin2020slightly}.
The bounds of \cite{karlin2020slightly} are doubly exponential in the number of variables, but they are still useful because they are independent of degree.
The bound we achieve here (see Theorem \ref{thm:main-TSP}, and also Theorem \ref{thm:main_nh}) is not only independent of degree, but it is also singly exponential in the number of variables.
Such an improvement was conjectured to be possible by the authors of \cite{karlin2020slightly} (see Section 1.2.2 in \cite{karlin2020slightly}), and thus our new bound will imply further improvement to the approximation factor for metric TSP.
We believe that our bounds here have the potential to imply similar improvements to other approximation algorithm results.


\section{Main Results}


Our first main result is a sharp improvement of the Linial-Samorodnitsky-Wigderson bound used in \cite{linial2000deterministic} to give a deterministic exponential approximation to the permanent of a matrix with non-negative entries.

\begin{theorem}[= Corollary \ref{cor:hstab-cap-bound}, Main homogeneous capacity bound] \label{thm:main}
    Let $p(x)$ be a homogeneous real stable polynomial of degree $n$ in $n$ variables.
    If $p(\mathbf{1}) = 1$ and $\|\mathbf{1} - \nabla p(\mathbf{1})\|_1 < 2$, then
    \[
        \inf_{x_1,\ldots,x_n > 0} \frac{p(x)}{x_1 \cdots x_n} \geq \left(1 - \frac{\|\mathbf{1} - \nabla p(\mathbf{1})\|_1}{2}\right)^n.
    \]
\end{theorem}

\noindent
The bound of Linial-Samorodnitsky-Wigderson took a similar form, but with the more restrictive assumption that $\|\mathbf{1} - \nabla p(\mathbf{1})\|_2 < \frac{1}{\sqrt{n}}$.
The fact that our assumption that $\|\mathbf{1} - \nabla p(\mathbf{1})\|_1 < 2$ is as strong as possible is an immediate consequence of our next result.

\begin{theorem}[= Proposition \ref{prop:hall-cap-norm}, Marginals which imply positive permanent] \label{thm:main-marginals}
    Given a non-negative vector $\alpha \in \R^n$, let $\Mat_n(\alpha)$ be the set of all $n \times n$ matrices $A$ with non-negative entries such that the row sums of $A$ are all equal to 1 and the column sums of $A$ are equal to $\alpha_1,\ldots,\alpha_n$.
    The following are equivalent.
    \begin{enumerate}
        \item The permanent of every matrix in $\Mat_n(\alpha)$ is strictly positive.
        \item $\|\bm{1} - \alpha\|_1 < 2$.
        \item $\sum_{i \in F} \alpha_i > |F| - 1$ for all $F \subseteq [n]$.
    \end{enumerate}
\end{theorem}

\noindent
Note that the main capacity bound given above can be seen as an analytic version of this result.

We also prove another bound for non-homogeneous real stable polynomials, which leads to a strong improvement of the capacity bound of Karlin-Klein-Oveis Gharan used in \cite{karlin2020slightly} to give an improved approximation factor for metric TSP.
The key bound in \cite{karlin2020slightly} is doubly exponential in the number of variables; our result here leads to a simply exponential bound (see Corollary \ref{cor:imp_metric_TSP}).

\begin{theorem}[= Corollary \ref{cor:imp-cap-bound}, Main non-homogeneous capacity bound] \label{thm:main_nh}
    Let $p(x)$ be a real stable polynomial in $n$ variables, and fix any $\kappa \in \Z^n$ with non-negative entries.
    If $p(\mathbf{1}) = 1$ and $\|\kappa - \nabla p(\mathbf{1})\|_1 < 1$, then
    \[
        \inf_{x_1,\ldots,x_n > 0} \frac{p(x)}{x_1^{\kappa_1} \cdots x_n^{\kappa_n}} \geq \left(1 - \|\kappa - \nabla p(\mathbf{1})\|_1\right)^n.
    \]
\end{theorem}

\noindent
Our main bounds are proven first for a specific class of stable polynomials: products of linear forms (see Theorem \ref{thm:prod-bound}).
This class is easier to work with due to its intimate relationship with the space of matrices.
To transfer the bound from products of linear forms to stable polynomials more generally, we use a new technique called \emph{productization}.
This bound transfer process is very general, which implies this productization technique is of independent interest.
We therefore state it as a result in its own right.

\begin{theorem}[= Theorem \ref{thm:hstab-prod}, Productization of real stable polynomials] \label{thm:main-productization}
    Let $p$ be a real stable homogeneous polynomial of degree $d$ in $n$ variables such that $p(\bm{1}) = 1$ and $\nabla p(\bm{1}) = \alpha$.
    For any $y \in \R^n$ in the positive orthant, there exists a product of linear forms $f(x) := \prod_{i=1}^d \sum_{j=1}^n a_{ij} x_j$ for which $a_{ij} \geq 0$ such that $f(\bm{1}) = 1$, $\nabla f(\bm{1}) = \alpha$, and $p(y) = f(y)$.
\end{theorem}

\noindent
Note that while we are able to use this technique to extend results to stable polynomials, a similar productization result for other classes of log-concave polynomials remains an open problem.

Beyond the original application to matrix scaling for approximating the permanent, such capacity bounds have also played an important role in the recent work on operator scaling and its generalizations and applications \cite{garg2018algorithmic}.
In a similar way, we can use our capacity bound to give a new scaling algorithm for stable polynomials in full generality.

\begin{theorem}[= Proposition \ref{prop:section-scaling-algo}, Stable polynomial scaling algorithm] \label{thm:main-scaling-algo}
    Let $p$ be a real stable homogeneous polynomial of degree $n$ in $n$ variables for which the coefficient of $x_1x_2 \cdots x_n$ is positive.
    There is an iterative algorithm for approximating, monotonically from above, the value of
    \[
        \inf_{x_1,\ldots,x_n > 0} \frac{p(x)}{x_1 \cdots x_n}.
    \]
    %
    For a given polynomial $p$, the rate of convergence is $O(1/t)$ where $t$ is the number of iterations.
\end{theorem}

\noindent
Note that the iterative algorithm we define is different, and easier to implement, than the iterative algorithm given by Bregman projections.
See Remark \ref{rem:bregman} for more details.

As discussed above, our main non-homogeneous bound improves the one used recently in \cite{karlin2020slightly} to give an improved approximation factor for metric TSP (see Corollary \ref{cor:imp_metric_TSP}).
We state this bound now, leaving further discussion of the relevant notation to Section \ref{sec:metric_TSP}.
Note that this bound is singly exponential in the number of random variables, in contrast to the doubly exponential bound of \cite{karlin2020slightly}.

\begin{theorem}[= Corollary \ref{cor:imp_metric_TSP} and Corollary \ref{cor:imp_metric_TSP_allones}, Application to metric TSP] \label{thm:main-TSP}
    Let $\mu$ be a strongly Rayleigh distribution on $\{0,1\}^m$, and let $X$ be a random variable distributed according to $\mu$.
    Let $S_1 \sqcup \ldots \sqcup S_n = [m]$ be a partition of $[m]$, and define random variables $A_1, \ldots, A_n$ via
    \[
        A_i := \sum_{s \in S_i} X_s \qquad \forall i \in [n].
    \]
    If $\left\|\mathbb{E}[A]-\kappa\right\|_1 \leq 1 - \epsilon$, then
    \[
        \mathbb{P}[\forall i : A_i = \kappa_i] \geq \epsilon^n \prod_{\kappa_i > 0} \frac{1}{e\sqrt{\kappa_i}}.
    \]
    In particular if $\kappa = \mathbf{1}$ and $\left\|\mathbb{E}[A]-\mathbf{1}\right\|_1 \leq 1 - \epsilon$, then
    \[
        \mathbb{P}[\forall i : A_i = 1] \geq \left(\frac{\epsilon}{e}\right)^n.
    \]
\end{theorem}


\section{Roadmap}

The remainder of the paper proceeds in a somewhat non-linear fashion, but each section (or subsection) gives explication of a single concept or result.
A description of each such section is given as follows.
\begin{itemize}
    \item Section \ref{sec:preliminaries}: Basic notation for polynomials and matrices, and some standard results regarding capacity and polynomials.
    \item Section \ref{sec:cap-bound-prod}: Proof of the main homogeneous capacity bound in the case of products of linear forms.
    \item Section \ref{sec:imp-bound-prod}: Proof of the main improved non-homogeneous capacity bound in the case of products of linear forms.
    \item Section \ref{sec:prod-hstab}: Proof of the productization result for stable polynomials.
    \item Section \ref{sec:metric_TSP}: Explication of the connection between our bound and the recent work on metric TSP.
    \item Section \ref{sec:algos}: Discussion of the algorithms for stable polynomial scaling, and for computing the minimum capacity value for a given gradient value.
    \item Section \ref{sec:other-bounds}: Various other bounds and discussion for more general classes of polynomials.
\end{itemize}

\section{Preliminaries} \label{sec:preliminaries}

\subsection{Notation} \label{sec:notation}

We let $\R,\R_+,\Z,\Z_+,\N,\C$ denote the reals, non-negative reals, integers, non-negative integers, positive integers, and complex numbers respectively.
We further let $\mathbb{K}^d[x] = \mathbb{K}^d[x_1,\ldots,x_n]$ denote the set of homogeneous polynomials of degree $d$ in $n$ variables with coefficients in $\mathbb{K}$.
For a polynomial $p$ in $n$ variables, the \emph{support} of $p$, denoted $\supp(p)$, is the set of all $\mu \in \Z_+^n$ such that $x^\mu$ has non-zero coefficient in $p$.
Further, the \emph{Newton polytope} of $p$, denoted $\Newt(p)$, is the convex hull of $\supp(p)$.
We also denote $\|\alpha\|_1 := \sum_{i=1}^n |\alpha_i|$ for $\alpha \in \R_+^n$ as usual.
We now define all of the various classes of matrices and polynomials we will consider.

\begin{definition}
    Given an $\alpha \in \R_+^n$ with $\|\alpha\|_1 = n$, we define $\Mat_n(\alpha)$ to be the set of $n \times n$ matrices $A$ with non-negative entries such that the row sums of $A$ are all 1 and the column sums of $A$ are given by $\alpha$.
\end{definition}

\begin{definition}
    Given $p \in \R_+^d[x_1,\ldots,x_n]$, we say that $p$ is \emph{real stable} if $p(x) = p(x_1,\ldots,x_n) \neq 0$ whenever $x_1,\ldots,x_n$ are all in the complex upper half-plane.
\end{definition}

\begin{definition}
    Given $p \in \R_+^d[x_1,\ldots,x_n]$, we say that $p$ is \emph{strongly log-concave} if $\nabla_{v_1} \cdots \nabla_{v_k} p$ is either identically zero or log-concave in the positive orthant for all $k \geq 0$ and all choices of $v_1, \ldots, v_k \in \R_+^n$. (These polynomials also go by the names \emph{completely log-concave} and \emph{Lorentzian}; see \cite{anari2018log} and \cite{branden2019lorentzian}.)
\end{definition}

\begin{definition}
    For $n \in \N$ and $\alpha \in \R_+^n$, we define the following classes of polynomials, ordered by inclusion:
    \begin{enumerate}
        \item $\Prod_n(\alpha)$ is the set of all polynomials of the form $p(x) = \prod_{i=1}^n (Ax)_i$, where $A \in \Mat_n(\alpha)$. Note that $p(\mathbf{1}) = 1$ and $\nabla p(\mathbf{1}) = \alpha$ for all such polynomials.
        \item $\HStab_n(\alpha)$ is the set of all real stable polynomials in $\R_+^n[x_1,\ldots,x_n]$ for which $p(\mathbf{1}) = 1$ and $\nabla p(\mathbf{1}) = \alpha$.
        \item $\SLC_n(\alpha)$ is the set of all strongly log-concave polynomials in $\R_+^n[x_1,\ldots,x_n]$ for which $p(\mathbf{1}) = 1$ and $\nabla p(\mathbf{1}) = \alpha$.
        \item $\LC_n(\alpha)$ is the set of all polynomials in $\R_+^n[x_1,\ldots,x_n]$ which are log-concave in the open positive orthant and for which $p(\mathbf{1}) = 1$ and $\nabla p(\mathbf{1}) = \alpha$.
    \end{enumerate}
\end{definition}

\noindent
We also give a special name to such matrices and polynomials whenever $\alpha = \mathbf{1}$.

\begin{definition}
    We refer to matrices in $\Mat_n(\mathbf{1})$ as \emph{doubly stochastic}. Similarly, if $p \in \R_+^n[x_1,\ldots,x_n]$ such that $p(\mathbf{1}) = 1$ and $\nabla p(\mathbf{1}) = \mathbf{1}$, then we say $p$ is \emph{doubly stochastic}.
\end{definition}

\noindent
Finally, we define the key quantity we study in this note.

\begin{definition}
    Given a polynomial $p \in \R_+^n[x_1,\ldots,x_n]$, we define the \emph{capacity} of $p$ as
    \[
        \cpc_\mathbf{1}(p) := \inf_{x > 0} \frac{p(x)}{x^\mathbf{1}} = \inf_{x_1,\ldots,x_n > 0} \frac{p(x)}{x_1 \cdots x_n}.
    \]
\end{definition}

\subsection{Basic Results} \label{sec:basic-results}

We state here a few standard basic results concerning polynomials, matrices, and capacity.

\begin{proposition}
    Polynomials of the form $\prod_{i=1}^n (Ax)_i$ for a given matrix $A$ with non-negative coefficients are homogeneous real stable, and homogeneous real stable polynomials are strongly log-concave.
\end{proposition}

\begin{proposition}[see \cite{gurvits2009multivariate}] \label{prop:supp-Newt}
    For any strongly log-concave $p \in \R_+^d[x_1,\ldots,x_n]$ and any $\mu \in \Z_+^n$, we have that $\mu \in \supp(p)$ if and only if $\mu \in \Newt(p)$.
\end{proposition}

\begin{proposition}[Symmetric exchange; see \cite{branden2019lorentzian}, Section 3.3] \label{prop:symm-exch}
    Let $p \in \R_+^d[x_1,\ldots,x_n]$ be strongly log-concave, and let $\mu,\nu \in \supp(p)$ such that $\mu_i > \nu_i$ for some $i \in [n]$. Then there exists $j \in [n]$ such that $\mu_j < \nu_j$ and $(\mu-\delta_i+\delta_j), (\nu+\delta_i-\delta_j) \in \supp(p)$.
\end{proposition}

\begin{proposition}[\cite{gurvits2008van}] \label{prop:ds-cap}
    If $p$ is doubly stochastic then $\cpc_\mathbf{1}(p) = 1$. In particular, if $A$ is doubly stochastic then $\cpc_\mathbf{1}(\sum_{i=1}^n (Ax)_i) = 1$. More generally, if $p(\mathbf{1}) = 1$ and $\nabla p(\mathbf{1}) = \alpha$, then $\cpc_\alpha(p) = 1$.
\end{proposition}

\begin{proposition} \label{prop:cap-Newt}
    Given a polynomial $p \in \R_+^n[x_1,\ldots,x_n]$, we have that $\cpc_\mathbf{1}(p) > 0$ if and only if $\mathbf{1} \in \Newt(p)$.
\end{proposition}

\begin{corollary} \label{cor:cap-per}
    Given an $n \times n$ matrix $A$ with non-negative entries, the following are equivalent:
    \begin{enumerate}
        \item $\per(A) = 0$.
        \item $\cpc_\mathbf{1}(\prod_{i=1}^n (Ax)_i) = 0$.
        \item Up to permutation, the bottom-left $i \times j$ block of $A$ is 0 for some $i+j > n$.
    \end{enumerate}
\end{corollary}
\begin{proof}
    $(1) \iff (2)$. Let us denote $p(x) := \prod_{i=1}^n (Ax)_i$.
    Recall that
    \[
        \per(A) = \partial_{x_1} \cdots \partial_{x_n} p.
    \]
    From this expression, it is clear to see that $\per(A) > 0$ iff $\mathbf{1} \in \supp(p)$.
    Further, since $p$ is a real stable polynomial, we know that $\mathbf{1} \in \supp(p)$ iff $\mathbf{1} \in \Newt(p)$ by Proposition \ref{prop:supp-Newt}.
    Finally, $\cpc_\mathbf{1}(p) > 0$ iff $\mathbf{1} \in \Newt(p)$ by Proposition \ref{prop:cap-Newt}.
    
    $(1) \iff (3)$. Follows from Hall's marriage theorem.
\end{proof}

\begin{lemma} \label{lem:linear-cap}
    For any $c \in \R_+^n$, we have
    \[
        \cpc_{\mathbf{1}}((c \cdot x)^n) = n^n \prod_{i=1}^n c_i.
    \]
\end{lemma}

\section{Capacity Bound for Product Polynomials} \label{sec:cap-bound-prod}

In this section we prove the main result (Theorem \ref{thm:main}) for polynomials in $\Prod_n(\alpha)$.
To simplify notation, we define the following for $n \in \N$ and $\alpha \in \R_+^n$:
\[
    L_n(\alpha) \equiv L_n^{\Prod}(\alpha) := \min_{p \in \Prod_n(\alpha)} \cpc_\mathbf{1}(p) \qquad \text{and} \qquad L_n^{\SLC}(\alpha) := \min_{p \in \SLC_n(\alpha)} \cpc_\mathbf{1}(p).
\]
Before proving the result, we need Proposition \ref{prop:hall-cap-norm}, which has some resemblance to Corollary \ref{cor:cap-per}.
In particular note that both results give equivalent conditions for capacity bounds and Hall-like properties.
First though, we need a few lemmas.

\begin{definition}
    Given a polynomial $p \in \R_+^n[x_1,\ldots,x_n]$, we say that $p$ is a \emph{Hall polynomial} if for all $S \subseteq [n]$ we have $\deg_S(p) \geq |S|$ where $\deg_S(p)$ is the total degree of $p$ involving variables with index in $S$.
\end{definition}

\begin{lemma} \label{lem:Hall-SLC}
    For any $p \in \R_+^n[x_1,\ldots,x_n]$, if $\cpc_\mathbf{1}(p) > 0$ then $p$ is a Hall polynomial. If $p$ is strongly log-concave, then these conditions are equivalent.
\end{lemma}
\begin{proof}
    We prove the contrapositive of the first statement.
    Let $S \subseteq [n]$ be such that $\deg_S(p) < |S|$.
    So for every degree vector $v$ which shows up in $p$ we have
    \[
        \sum_{i \in S} v_i < |S|,
    \]
    and therefore the same inequality holds for every $v \in \Newt(p)$.
    In particular, $\mathbf{1} \not\in \Newt(p)$ and so $\cpc_\mathbf{1}(p) = 0$ by Proposition \ref{prop:cap-Newt}.
    
    Now suppose that $p$ is strongly log-concave and that $p$ is a Hall polynomial.
    We set out to show that $q := \left.\partial_{x_n} p\right|_{x_n=0}$ is a Hall polynomial.
    Fix $S \subseteq [n-1]$, and let $\mu \in \supp(p)$ be such that $\deg_S(x^\mu) \geq |S|$. We have three cases.
    
    \textbf{Case 1:} $\deg_n(x^\mu) \geq 1$.
    Let $\nu \in \supp(p)$ be such that $\deg_{[n-1]}(x^\nu) \geq n-1$, so that $\deg_n(x^\nu) \leq 1$.
    By applying symmetric exchange (Proposition \ref{prop:symm-exch}) from $\mu$ to $\nu$, there exists $\mu' \in \supp(p)$ such that $\deg_S(x^{\mu'}) \geq |S|$ and $\deg_n(x^{\mu'}) = 1$.
    This implies $\deg_S(q) \geq |S|$.
    
    \textbf{Case 2:} $\deg_n(x^\mu) = 0$ and $\deg_S(x^\mu) > |S|$.
    Let $\nu \in \supp(p)$ be such that $\deg_n(x^\nu) \geq 1$.
    By applying symmetric exchange from $\nu$ to $\mu$, there exists $\mu' \in \supp(p)$ such that $\deg_S(x^{\mu'}) \geq |S|$ and $\deg_n(x^{\mu'}) = 1$.
    This implies $\deg_S(q) \geq |S|$.
    
    \textbf{Case 3:} $\deg_n(x^\mu) = 0$ and $\deg_S(x^\mu) = |S|$.
    Let $\nu \in \supp(p)$ be such that $\deg_{S \cup \{n\}}(x^\nu) \geq |S|+1$.
    Letting $T := [n-1] \setminus S$, we have that $\deg_T(x^\mu) = |T|+1$ and $\deg_T(\nu) \leq |T|$.
    Apply symmetric exchange from $\mu$ to $\nu$, choosing indices from $T$ to remove from $\mu$, until we have $\mu' \in \supp(p)$ such that $\deg_j(x^{\mu'}) \leq \deg_j(x^\nu)$ for all $j \in T$.
    This implies $\deg_T(x^{\mu'}) \leq \deg_T(x^\nu) \leq |T| = n-1-|S|$ and $\deg_S(x^{\mu'}) \geq |S|$.
    Therefore either $\deg_n(x^{\mu'}) \geq 1$ or $\deg_S(x^{\mu'}) > |S|$, and so one of the previous two cases can be applied to $x^{\mu'}$.
    
    In any case we have $\deg_S(q) \geq |S|$, and therefore $q$ is a Hall polynomial.
    Since $q$ is also strongly log-concave, we inductively have $\cpc_\mathbf{1}(q) > 0$.
    By Euler's identity, this implies
    \[
        \cpc_\mathbf{1}(p) \geq \frac{1}{n} \sum_{i=1}^n \cpc_\mathbf{1}(x_i \cdot \left.\partial_{x_i} p\right|_{x_i=0}) > 0.
    \]
\end{proof}

\noindent
The following Van der Waerden-type bound for $n$-variate, $n$-homogeneous strongly log-concave polynomials $p$ was proven by Gurvits in \cite{gurvits2009polynomial, gurvits2009multivariate}:
\[
    \partial_{x_1} \cdots \partial_{x_n} p \geq \frac{n!}{n^n} \cpc_{\bm{1}}(p).
\]

\begin{lemma} \label{lem:nabla-deg-bound}
    Let $p \in \R_+[x_1,\ldots,x_n]$ be such that $p(\mathbf{1}) = 1$, and let $\nabla p(\mathbf{1}) = \alpha$. For all $S \subseteq [n]$, we have
    \[
        \sum_{i \in S} \alpha_i \leq \deg_S(p).
    \]
\end{lemma}
\begin{proof}
    By plugging in $x_i = 1$ for all $i \not\in S$, we may assume that $S = [n]$.
    (Note that we are not assuming $p$ is homogeneous.)
    Letting $P$ be the homogenization of $p$, we have that $\nabla P(\mathbf{1}) = (\alpha_1,\ldots,\alpha_n,\beta)$ and $\deg(P) = \deg(p)$.
    Therefore
    \[
        \beta + \sum_{i=1}^n \alpha_i = \deg(P) \implies \sum_{i=1}^n \alpha_i \leq \deg(p).
    \]
    This completes the proof.
\end{proof}

\begin{proposition} \label{prop:hall-cap-norm}
    Given $\alpha \in \R_+^n$ such that $\sum_i \alpha_i = n$, the following are equivalent:
    \begin{enumerate}
        \item $\|\mathbf{1}-\alpha\|_1 < 2$.
        \item $L_n^{\SLC}(\alpha) > 0$.
        \item $L_n(\alpha) \equiv L_n^{\Prod}(\alpha) > 0$.
        \item $\sum_{i \in F} \alpha_i > |F|-1$ for all $F \subseteq [n]$.
        \item Every $p \in \SLC_n(\alpha)$ is a Hall polynomial.
    \end{enumerate}
\end{proposition}
\begin{proof}
    Let $\delta := 1-\alpha$, so that $\sum_i \delta_i = 0$. For any $F \subseteq [n]$, let $F = F_+ \sqcup F_-$ such that $\delta_i \geq 0$ for $i \in F_+$ and $\delta_i < 0$ for $i \in F_-$.
    
    $(1) \implies (4)$. For any $F$, we have
    \[
        \sum_{i \in F} \alpha_i = |F| - \sum_{i \in F} \delta_i = |F| + \sum_{i \in F_-} \delta_i - \sum_{i \in F_+} \delta_i.
    \]
    Since $\sum_i \delta_i = 0$ and $\sum_i |\delta_i| < 2$, we have that $\sum_{i \in F_+} \delta_i < 1$. This implies
    \[
        \sum_{i \in F} \alpha_i = |F| + \sum_{i \in F_-} \delta_i - \sum_{i \in F_+} \delta_i > |F| - 1.
    \]
    
    $(4) \implies (1)$. Letting $F = [n]$ with $F_+,F_-$ defined as above, we have
    \[
        |F_+| - \sum_{i \in F_+} \delta_i = \sum_{i \in F_+} \alpha_i > |F_+| - 1 \implies \sum_{\delta_i \geq 0} \delta_i < 1.
    \]
    Since $\sum_i \delta_i = 0$, this implies
    \[
        \|1-\alpha\|_1 = \sum_i |\delta_i| = \sum_{\delta_i \geq 0} \delta_i - \sum_{\delta_i < 0} \delta_i < 2.
    \]
    
    $(2) \implies (3)$. Trivial.
    
    $(3) \implies (4)$. So as to get a contradiction, suppose there is some $k \in [n]$ such that $\alpha_1 + \cdots + \alpha_k \leq k-1$.
    We now construct a matrix $A \in \Mat_n(\alpha)$ such that $\cpc_\mathbf{1}(\prod_i (Ax)_i) = 0$.
    Let $A_1$ be a $(k-1) \times k$ matrix with column sums $\alpha_1,\ldots,\alpha_k$ and row sums all equal to $\beta := \frac{\alpha_1 + \cdots + \alpha_k}{k-1}$.
    Since $\beta \leq 1$, we can define
    \[
        A := \begin{bmatrix}
            A_1 & * \\
            0 & *
        \end{bmatrix} \in \Mat_n(\alpha),
    \]
    where the bottom-left $(n-k+1) \times k$ block of $A$ is 0.
    Since $(n-k+1)+k > 0$, Corollary \ref{cor:cap-per} implies $\cpc_\mathbf{1}(\prod_i (Ax)_i) = 0$.
    Therefore $L_n(\alpha) = 0$.
    
    $(4) \implies (5)$. Follows from Lemma \ref{lem:nabla-deg-bound}.
    
    $(5) \implies (2)$. Follows from Lemma \ref{lem:Hall-SLC}.
    
\end{proof}

\noindent
We now state the main result of this section: a lower bound on the capacity of polynomials in $\Prod_n(\alpha)$.

\begin{theorem} \label{thm:prod-bound}
    Fix $n \in \N$, $\alpha \in \R_+^n$, and $p \in \Prod_n(\alpha)$. If $\|1-\alpha\|_1 < 2$, then
    \[
        \cpc_\mathbf{1}(p) \geq \left(1 - \frac{\|1-\alpha\|_1}{2}\right)^n.
    \]
\end{theorem}
\begin{proof}
    Define $\delta := 1-\alpha$.
    We just need to prove
    \[
        L_n(\alpha) \geq \left(1 - \frac{\|\delta\|_1}{2}\right)^n.
    \]
    Define
    \[
        S := \{(\gamma, D) \in \R_+ \times \Mat_n(\mathbf{1}) ~:~ A - \gamma D \geq 0 ~~ \text{entrywise}\},
    \]
    and further define $(\gamma_0, D_0) \in S$ to be such that $\gamma_0$ is maximized.
    (This maximum exists by compactness of $\Mat_n(\mathbf{1})$, the Birkhoff polytope.)
    Now consider the matrix
    \[
        M = \frac{A - \gamma_0 D_0}{1 - \gamma_0},
    \]
    which is an element of $\Mat_n(\tilde{\alpha})$ for $\tilde{\alpha} = \frac{\alpha-\gamma_0}{1-\gamma_0}$.
    We now show that $\per(M) = 0$.
    If not, then there is some permutation matrix $P$ and some $\epsilon > 0$ such that
    \[
        A - (\gamma_0 + \epsilon) \cdot \frac{\gamma_0 D_0 + \epsilon P}{\gamma_0 + \epsilon}
    \]
    is entrywise non-negative.
    Since $\frac{\gamma_0 D_0 + \epsilon P}{\gamma_0 + \epsilon} \in \Mat_n(\mathbf{1})$, this contradicts the maximality of $\gamma_0$.
    So in fact $\per(M) = 0$, and therefore $L_n(\tilde{\alpha}) = 0$ by Corollary \ref{cor:cap-per}.
    By Proposition \ref{prop:hall-cap-norm}, this implies
    \[
        \left\|\frac{\delta}{1-\gamma_0}\right\|_1 = \|1 - \tilde{\alpha}\|_1 \geq 2 \implies \gamma_0 \geq 1 - \frac{\|\delta\|_1}{2}.
    \]
    Since $D_0 \in \Mat_n(\mathbf{1})$, we have that $\cpc_\mathbf{1}(\prod_i (D_0x)_i) = 1$ by Proposition \ref{prop:ds-cap}.
    The fact that $A \geq \gamma_0 D_0$ entrywise then implies
    \[
        \cpc_\mathbf{1}\left[\prod_{i=1}^n (Ax)_i\right] \geq \cpc_\mathbf{1}\left[\prod_{i=1}^n (\gamma_0 D_0x)_i\right] \geq \gamma_0^n \geq \left(1 - \frac{\|\delta\|_1}{2}\right)^n.
    \]
\end{proof}

\noindent
We finally state the following corollary of Proposition \ref{prop:hall-cap-norm} and Theorem \ref{thm:prod-bound}, which gives a similar result for the 2-norm instead of the 1-norm.

\begin{corollary} \label{cor:2-norm-cap}
    Fix $n \in \N$, $\alpha \in \R_+^n$, and $p \in \Prod_n(\alpha)$. If $\|1 - \alpha\|_2 < \frac{2}{\sqrt{n}}$, then
    \[
        \cpc_\mathbf{1}(p) \geq \left(1 - \frac{\sqrt{n} \cdot \|1 - \alpha\|_2}{2}\right)^n.
    \]
\end{corollary}
\begin{proof}
    Follows from $\|x\|_1 \leq \sqrt{n} \|x\|_2$ for $x \in \R^n$.
\end{proof}

\noindent
The corollary is a sharp improvement of the following similar inequality proved in \cite{linial2000deterministic} in the case that $\|1-\alpha\|_2 < \frac{1}{\sqrt{n}}$:
\[
    L_{n}(\alpha) \geq \left(1 - \sqrt{n} \cdot \|1 - \alpha\|_2\right)^{n}.
\]
This last inequality also plays a key role in the recent work on the operator scaling and its generalizations and applications, see \cite{garg2018algorithmic}.

\section{Improved Capacity Bound for Non-Homogeneous Product Polynomials} \label{sec:imp-bound-prod}

In this section we prove the main result (Theorem \ref{thm:main_nh}) for non-homogeneous product polynomials. For this section, we utilize slightly different notation to handle non-homogeneous polynomials.

\begin{definition}
    For $n,d \in \N$ and $\alpha \in \R_+^n$, we define the following:
    \begin{enumerate}
        \item $\NHMat_n^d(\alpha)$ is the set of all $d \times (n+1)$ matrices with row sums all equal to 1 and column sums equal to $\alpha_1,\ldots,\alpha_n,d-\|\alpha\|_1$.
        \item $\NHProd_n^d(\alpha)$ is the set of all polynomials of the form
        \[
            p(x) = \prod_{i=1}^d \left(a_{i,n+1} + \sum_{j=1}^n a_{i,j} x_j\right),
        \]
        where $A \in \NHMat_n^d(\alpha)$. In this case, we call $p$ the \emph{polynomial associated to $A$}. Note that $p(\mathbf{1}) = 1$ and $\nabla p(\mathbf{1}) = \alpha$ for all such polynomials.
        \item $\NHStab_n^d(\alpha)$ is the set of all real stable polynomials in $\R_+[x_1,\ldots,x_n]$ of degree at most $d$ for which $p(\mathbf{1}) = 1$ and $\nabla p(\textbf{1})= \alpha$.
    \end{enumerate}
\end{definition}

\noindent
We also define the following for $n \in \N$ and $\alpha \in \R_+^n$:
\[
    L^{\NHProd}_n(\alpha;\kappa) := \min_{d \in \N} \min_{p \in \NHProd_n^d(\alpha)} \cpc_\kappa(p).
\]

To prove Theorem \ref{thm:main_nh} for polynomials in $\NHProd_n^d(\alpha)$, we use different techniques than we used in the previous section for homogeneous polynomials. Specifically, a log-concavity property of capacity means we only need to check the bound for the extreme points of $\NHMat_n^d(\alpha)$. An investigation of the extreme points yields an inductive argument which then proves the result.

To this end, we prove a basic result about bipartite forest graphs, which implies a useful property of the support of the polynomials associated to the extreme points of $\NHMat_n^d(\alpha)$.

\begin{lemma} \label{lem:left_leaves}
    Let $G$ be a bipartite forest on $m$ left vertices and $n$ right vertices such that $G$ has no vertices of degree 0. Then $G$ has at least $m-n+1$ left leaves.
\end{lemma}
\begin{proof}
    We prove this by induction on $m+n$, where the base case is any path graph. For this case, we have $m \in \{n-1,n,n+1\}$. If $m = n-1$ then $G$ has 0 left leaves, if $m = n$ then $G$ has 1 left leaf, and if $m = n+1$ then $G$ has two left leaves. Thus the desired result holds in this case.
    
    For the inductive step, $G$ is not a path graph. Let $v$ be any leaf of $G$, and construct a new graph $G'$ as follows. Let $v_0 := v$ and remove $v_0$ from $G_0 := G$ to create the graph $G_1$. Let $v_1$ be the one neighbor of $v_0$ in $G$. If $v_1$ is a leaf or vertex of degree 0 in $G_1$, then remove $v_1$ from $G_1$ to create the graph $G_2$. If $v_1$ was a vertex of degree 0, then stop and define $G' := G_1$. Otherwise let $v_2$ be the one neighbor of $v_1$, and continue this process inductively until $v_k$ is not a leaf or a vertex of degree 0 in $G_k$. Once the process stops, define $G' := G_k$. Note that $G'$ is a bipartite forest which has no vertices of degree 0, and $G'$ is non-empty since $G$ is not a path graph. Further note that every leaf of $G'$ is also a leaf of $G$ (that is, we have not created new leaves by the above vertex removal process). We now have two cases: $v$ is a left leaf of $G$, or $v$ is a right leaf of $G$.
    
    First suppose $v$ is a left leaf of $G$. Then for some $i$, $G'$ has $m-i$ left vertices and at most $n-i+1$ right vertices. Thus by induction, the number of left leaves of $G$ is at least
    \[
        1 + (m-i) - (n-i+1) + 1 = m - n + 1
    \]
    since $v$ is a left leaf. Thus the result holds in this case.
    
    Next suppose $v$ is a right leaf of $G$. Then for some $i$, $G'$ has $m-i$ left vertices and at most $n-i$ right vertices. Thus by induction, the number of left leaves of $G$ is at least
    \[
        (m-i) - (n-i) + 1 = m - n + 1.
    \]
    Thus the result holds in this case as well.
\end{proof}

\begin{lemma} \label{lem:forest_extreme}
    Any extreme point of $\NHMat_n^d(\alpha)$ has support given by a bipartite forest on $d$ left vertices and $n+1$ right vertices.
\end{lemma}
\begin{proof}
    Let $M$ be an extreme point of $\NHMat_n^d(\alpha)$, and suppose its bipartite support graph $G$ does not give a forest. Then $G$ must contain an even simple cycle. Group the edges of this cycle into two groups such that the odd edges make up one group, and the even edges make up the other (with any starting point). Add $\epsilon > 0$ to all matrix entries corresponding to even edges and subtract $\epsilon$ to all matrix entries corresponding to odd edges, to construct $M_+ \in \NHMat_n^d(\alpha)$. Do the same thing, but reverse the signs, to construct $M_- \in \NHMat_n^d(\alpha)$. Thus $M = \frac{M_+ + M_-}{2}$, contradicting the fact that $M$ is an extreme point.
\end{proof}

\noindent
With this, we now prove a basic capacity lemma which will serve as the main induction step of the proof of Theorem \ref{thm:main_nh} for product polynomials.

\begin{lemma} \label{lem:cap_induct}
    For $n \geq 1$, let $p \in \NHProd_n^d(\alpha)$ be the polynomial associated to a $d \times (n+1)$ matrix $M \in \NHMat_n^d(\alpha)$ such that column $n$ has exactly one non-zero entry, and fix $\kappa \in \Z_+^n$ such that $\|\alpha-\kappa\|_1 \leq 1-\epsilon$ for some $\epsilon > 0$. Then there exists $q \in \NHProd_{n-1}^{\delta}(\beta)$ such that
    \[
        \cpc_\kappa(p) \geq \epsilon \cdot \cpc_\gamma(q),
    \]
    where $\gamma = (\kappa_1,\ldots,\kappa_{n-1}) \in \Z_+^{n-1}$, $\delta$ is either $d-1$ or $d$, and $\beta \in \R_+^{n-1}$ is such that $\|\beta-\gamma\|_1 \leq 1 - \epsilon$.
\end{lemma}
\begin{proof}
    Let $i_0$ be the row index of the one non-zero entry in column $n$. Since all row sums of $M$ are equal to 1, $m_{i_0,n} \leq 1$ and thus $\alpha_n \leq 1$. Since $\|\alpha-\kappa\|_1 \leq 1 - \epsilon < 1$, this implies $\kappa_n$ is either 0 or 1. We now deal with these two cases separately.
    
    First we handle the $\kappa_n = 1$ case. In this case, we define $q(x) = q(x_1,\ldots,x_{n-1})$ via
    \[
        p(x) = \left(m_{i_0,n+1} + \sum_{j=1}^n m_{i_0,j} x_j\right) \cdot q(x).
    \]
    That is, $q(x)$ is the polynomial associated to the matrix obtained by removing row $i_0$ of $M$. We then have
    \[
        \cpc_\kappa(p) \geq \cpc_{e_n}\left(m_{i_0,n+1} + \sum_{j=1}^n m_{i_0,j} x_j\right) \cdot \cpc_\gamma(q) = m_{i_0,n} \cdot \cpc_\gamma(q),
    \]
    where $e_n$ is the $n^{\text{th}}$ standard basis vector. Note that $q$ is the polynomial associated to a $(d-1) \times n$ matrix $A$ with row sums all equal to 1. Let $\beta$ be the vector of the first $n-1$ column sums of this matrix, so that $q \in \NHProd_{n-1}^{d-1}(\beta)$. Now note that
    \[
        m_{i_0,n} = 1 - |m_{i_0,n} - 1| \geq 1 - \|\alpha-\kappa\|_1 \geq 1 - (1-\epsilon) = \epsilon.
    \]
    And further, by the triangle inequality we have
    \[
        \|\beta - \gamma\|_1 = \sum_{j=1}^{n-1} |\alpha_j - m_{i_0,j} - \kappa_j| = -1 + \sum_{j=1}^n |\alpha_j - m_{i_0,j} - \kappa_j| \leq \|\alpha - \kappa\|_1 - 1 + \sum_{j=1}^n m_{i_0,j} \leq \|\alpha - \kappa\|_1.
    \]
    This implies the desired result.
    
    Next we handle the $\kappa_n = 0$ case. In this case, we define $q_0(x) = q_0(x_1,\ldots,x_{n-1})$ via
    \[
        q_0(x) = p(x_1,\ldots,x_{n-1},0),
    \]
    so that
    \[
        \cpc_\kappa(p) = \cpc_\gamma(q_0).
    \]
    Note that $q_0$ is the polynomial associated to a $d \times n$ matrix $A$, where all row sums of $A$ are equal to 1 except for row $i_0$ which has row sum equal to $1 - m_{i_0,n} = 1 - \alpha_n$. Note that $\alpha_n \leq \|\alpha - \kappa\|_1 \leq 1 - \epsilon$ implies $1 - \alpha_n \geq \epsilon > 0$. We now construct a new matrix $B \in \NHMat_{n-1}^d(\beta)$ by dividing row $i_0$ of $A$ by $1 - \alpha_n$. Defining $q(x)$ to be the polynomial associated to the matrix $B$, we have
    \[
        \cpc_\kappa(p) = \cpc_\gamma(q_0) = (1 - \alpha_n) \cdot \cpc_\gamma(q) \geq \epsilon \cdot \cpc_\gamma(q).
    \]
    Finally, by the triangle inequality we compute
    \[
    \begin{split}
        \|\beta - \gamma\|_1 &= \sum_{j=1}^{n-1} \left|\alpha_j - m_{i_0,j} + \frac{m_{i_0,j}}{1-\alpha_n} - \kappa_j\right| = \sum_{j=1}^{n-1} \left|\alpha_j + \frac{m_{i_0,j} \cdot \alpha_n}{1-\alpha_n} - \kappa_j\right| \\
            &\leq \|\alpha - \kappa\|_1 - \alpha_n + \frac{\alpha_n}{1-\alpha_n} \sum_{j=1}^{n-1} m_{i_0,j} \leq \|\alpha - \kappa\|_1
    \end{split}
    \]
    This completes the proof.
\end{proof}

\noindent
The following result then allow us to extend capacity bounds associated to extreme points of $\NHMat_n^d(\alpha)$ to all polynomials in $\NHProd_n^d(\alpha)$.

\begin{lemma}[See Theorem 3.1 of \cite{barvinok2008enumerating}] \label{lem:log-concave_cap}
    Given $\kappa \in \Z_+^n$ and $\alpha \in \R_+^n$, let $\phi: \NHMat_n^d(\alpha) \to \R_+$ be the function which maps $M$ to $\cpc_\kappa(p)$ where $p$ is the polynomial associated to $M$. Then $\phi$ is log-concave on $\NHMat_n^d(\alpha)$.
\end{lemma}
\begin{proof}
    Fix any $A,B \in \NHMat_n^d(\alpha)$, and let $p,q,f$ be the polynomials associated to $A,B, \frac{A+B}{2}$ respectively. By the AM-GM inequality, we compute
    \[
    \begin{split}
        \cpc_\kappa(f) &= \inf_{x > 0} \frac{\prod_{i=1}^d \left(\frac{a_{i,n+1} + b_{i,n+1}}{2} + \sum_{j=1}^n \frac{a_{i,j} + b_{i,j}}{2} \cdot x_j\right)}{x^\kappa} \\
            &= \inf_{x > 0} \frac{\prod_{i=1}^d \left[\frac{\left(a_{i,n+1} + \sum_{j=1}^n a_{i,j} x_j\right) + \left(b_{i,n+1} + \sum_{j=1}^n b_{i,j} x_j\right)}{2}\right]}{x^\kappa} \\
            &\geq \inf_{x > 0} \left[\frac{\prod_{i=1}^d \left(a_{i,n+1} + \sum_{j=1}^n a_{i,j} x_j\right)}{x^\kappa} \cdot \frac{\prod_{i=1}^d \left(b_{i,n+1} + \sum_{j=1}^n b_{i,j} x_j\right)}{x^\kappa}\right]^{\frac{1}{2}} \\
            &\geq \left[\cpc_\kappa(p) \cdot \cpc_\kappa(q)\right]^{\frac{1}{2}}.
    \end{split}
    \]
\end{proof}

\noindent
With this, we are now ready to prove Theorem \ref{thm:main_nh} for product polynomials.

\begin{theorem} \label{thm:improved_cap_bound}
    For any $\kappa \in \Z_+^n$, $\alpha \in \R_+^n$, and $\epsilon > 0$, if $\|\alpha-\kappa\|_1 \leq 1-\epsilon$ then $\cpc_{\kappa}(p) \geq \epsilon^n$ for every $p \in \NHProd_n^d(\alpha)$. That is, $L_n^{\NHProd}(\alpha;\kappa) \geq \epsilon^n$ whenever $\|\alpha-\kappa\|_1 \leq 1-\epsilon$.
\end{theorem}
\begin{proof}
    We prove the desired result by induction on $(n,d)$ with lexicographical order. The base case is the case where $n=0$, which corresponds to $p \equiv 1$ for any $d$. In this case, $\kappa$ is the empty vector and $\cpc_{\kappa}(p) = 1 = \epsilon^0$, which implies the desired result. (See Example \ref{ex:thm_n1} for the $n=1$ case written out in detail.)
    
    Now for $n \geq 1$, fix any extreme point $M \in \NHMat_n^d(\alpha)$, and let $p$ be the polynomial associated to $M$. By Lemma \ref{lem:forest_extreme}, $M$ is the (weighted) bipartite adjacency matrix of a forest $G$.
    
    If any of the first $n$ columns of $M$ is the all-zeros vector, then $\|\alpha-\kappa\|_1 \leq 1 - \epsilon < 1$ implies the corresponding entry of $\kappa$ is 0. Thus we can remove this column from the matrix $M$, and the result follows by induction. If column $n+1$ of $M$ is the all-zeros vector, then $p$ is homogeneous of degree $d$ which implies $\|\alpha\|_1 = d$. The reverse triangle inequality and $\|\alpha - \kappa\|_1 < 1$ imply $\|\kappa\|_1 = d$. Thus we have
    \[
        \cpc_{\kappa}(p) = \inf_{x > 0} \frac{p(x_1,\ldots,x_n)}{x_1^{\kappa_1} \cdots x_n^{\kappa_n}} = \inf_{x > 0} \frac{p(\frac{x_1}{x_n},\ldots,\frac{x_{n-1}}{x_n},1)}{(\frac{x_1}{x_n})^{\kappa_1} \cdots (\frac{x_{n-1}}{x_n})^{\kappa_{n-1}}} = \cpc_{\gamma}(p(x_1,\ldots,x_{n-1},1)),
    \]
    where $\gamma = (\kappa_1,\ldots,\kappa_{n-1})$. Since we have $p(x_1,\ldots,x_{n-1},1) \in \NHProd_{n-1}^d(\beta)$ where $\beta = (\alpha_1,\ldots,\alpha_{n-1})$, the result follows by induction. Otherwise, $G$ is a forest on $d$ left vertices and $n+1$ right vertices such that no vertices of $G$ have degree 0.
    
    We first consider the case where $d \geq n+1$. Then $G$ has at least as many left vertices as right vertices. By Lemma \ref{lem:left_leaves}, $M$ has (at least) $d-n$ rows with exactly one non-zero entry. Letting $x^v$ be the polynomial associated to these rows of $M$, we have
    \[
        \cpc_{\kappa}(p) = \inf_{x > 0} \frac{x^v f(x)}{x^{\kappa}} = \cpc_{\kappa-v}(f).
    \]
    With this, we now redefine $\kappa$ to be $\kappa - v$, $\alpha$ to be $\alpha - v$, $p$ to be $f$, and $M$ to be the matrix in $\NHMat_n^n(\alpha-v)$ with the $d-n$ rows removed. The result in this case then follows by induction.
    
    Now we consider the case where $d \leq n$. Then $G$ has at least one more right vertex than left vertices. By Lemma \ref{lem:left_leaves}, $M$ has (at least) $n-d+2 \geq 2$ columns with exactly one non-zero entry. Thus one of these columns has index in $\{1,2,\ldots,n\}$, and we may assume by permuting the variables that it in fact is column $n$. By Lemma \ref{lem:cap_induct}, there exists $q \in \Prod_{n-1}^\delta(\beta)$ such that
    \[
        \cpc_{\kappa}(p) \geq \epsilon \cdot \cpc_{\gamma}(q),
    \]
    where $\gamma = (\kappa_1,\ldots,\kappa_{n-1}) \in \Z_+^{n-1}$, $\delta$ is either $d-1$ or $d$, and $\beta \in \R_+^{n-1}$ is such that $\|\beta-\gamma\|_1 \leq 1 - \epsilon$. By induction, this proves the result for polynomials associated to the extreme points of $\NHMat_n^d(\alpha)$. Applying Lemma \ref{lem:log-concave_cap} then completes the proof.
\end{proof}

\begin{example} \label{ex:thm_n1}
    We explicitly prove Theorem \ref{thm:improved_cap_bound} in the $n=1$ case. Fix any extreme point $M \in \NHMat_1^d(\alpha)$, and let $p$ be the polynomial associated to $M$. If either column of $M$ is an all-zeros vector, then either $p(x_1) = x_1^d$ or $p(x_1) \equiv 1$, and the result is trivial in this case. Otherwise, Lemma \ref{lem:left_leaves} implies $d-1$ rows of $M$ have exactly one non-zero entry. Thus for some $k \leq d-1$ we have
    \[
        p(x_1) = x_1^k (a x_1 + b),
    \]
    where $a+b = 1$ and $a + k = \alpha_1$. Since $\kappa_1 \in \Z_+$ and
    \[
        1 - \epsilon \geq |\alpha_1 - \kappa_1| = |a + k - \kappa_1|,
    \]
    we have that $k$ is equal to either $\kappa_1$ or $\kappa_1 - 1$. If $k = \kappa_1$, then
    \[
        \cpc_{\kappa_1}(p) = \inf_{x_1 > 0} \frac{x_1^k(ax_1+b)}{x_1^k} = b,
    \]
    and
    \[
        1 - \epsilon \geq |a + k - \kappa_1| = a \implies b = 1-a \geq \epsilon.
    \]
    If $k = \kappa_1 - 1$, then
    \[
        \cpc_{\kappa_1}(p) = \inf_{x_1 > 0} \frac{x_1^k(ax_1+b)}{x_1^{k+1}} = a,
    \]
    and
    \[
        1 - \epsilon \geq |a + k - \kappa_1| = 1 - a \implies a \geq \epsilon.
    \]
    Therefore the result holds in both cases. Applying Lemma \ref{lem:log-concave_cap} then implies Theorem \ref{thm:improved_cap_bound} in the $n=1$ case.
\end{example}

\section{Productization of Real Stable Polynomials} \label{sec:prod-hstab}

In this section we prove the productization result for $d$-homogeneous polynomials in $n$ variables.
This result immediately implies the main results (Theorem \ref{thm:main} and Theorem \ref{thm:main_nh}) for polynomials in $\HStab_n(\alpha)$ and $\NHStab_n^d(\alpha)$ as corollaries (see Corollary \ref{cor:hstab-cap-bound} and Corollary \ref{cor:imp-cap-bound}).

To actually prove the productization result, we need a way to associate matrices to polynomials.
For the case of $\alpha = \bm{1}$, this statement was conjectured by Gurvits in the slightly different form given below.
The conjecture was motivated by the case of determinantal polynomials, where the desired element of $\Mat_n(\bm{1})$ can be constructed from the matrices in the determinant.
We now state this result, the proof of which was told to us by Petter Br\"and\'en in personal correspondence.

\begin{theorem}[Br\"and\'en] \label{thm:roots-majorization}
    Fix $p \in \HStab_n(\bm{1})$, and let $\lambda(x)$ denote the roots of $f(t) = p(\bm{1}t - x)$ for any $x \in \R^n$. Then for any $x \in \R^n$, there exists some $D \in \Mat_n(\bm{1})$ such that $Dx = \lambda(x)$.
\end{theorem}
\begin{proof}
    To prove the result, we just need to prove the equivalent statement that $\lambda(x)$ is majorized by the vector $x$, denoted $\lambda(x) \prec x$.
    By Lemma 3.8 of \cite{bauschke2001hyperbolic}, we have that
    \[
        \lambda(\theta \cdot x + (1-\theta) \cdot y) \prec \theta \cdot \lambda(x) + (1-\theta) \cdot \lambda(y).
    \]
    Let $C$ be the set of all non-increasing vectors $x$ for which $\lambda(x)$ is majorized by $x$.
    For all $x,y \in C$, we then have
    \[
        \lambda(\theta \cdot x + (1-\theta) \cdot y) \prec \theta \cdot \lambda(x) + (1-\theta) \cdot \lambda(y) \prec \theta \cdot x + (1-\theta) \cdot y,
    \]
    which implies $C$ is a convex cone.
    
    Now let $y = (1,\ldots,1,0,\ldots,0)$, where there are $k$ ones and $n-k$ zeros.
    By the majorization inequalities, $\lambda(y) \prec y$ in this case is equivalent to $0 \leq \lambda_{\min}(y) \leq \lambda_{\max}(y) \leq 1$ and $\sum_i \lambda_i(y) = k$.
    To see that the roots of $f(t) = p(\bm{1}t-y)$ have these properties, note first that $t < 0$ implies $\bm{1}t-y$ has strictly negative entries, which implies $f(t) = p(\bm{1}t-y) \neq 0$.
    Similarly, $t > 1$ implies $\bm{1}t-y$ has strictly positive entries, which implies $f(t) = p(\bm{1}t-y) \neq 0$.
    Finally, since $p(\bm{1}) = 1$ implies $f$ is monic, we can compute the sum of the roots of $f(t)$ by negating the second from the highest coefficient of $f$.
    Equivalently, we compute
    \[
        -\left.\partial_t\right|_{t=0} p(\bm{1} - ty) = -\left.\sum_{i=1}^n -y_i \partial_{x_i} p(\bm{1} - ty)\right|_{t=0} = \sum_{i=1}^n y_i \partial_{x_i} p(\bm{1}) = k,
    \]
    since $\partial_{y_i}p(\bm{1}) = 1$ for all $i$. That is, $y \in C$ for all values of $k$.
    
    Since $C$ is a convex cone, this implies that every non-increasing vector in the positive orthant is an element of $C$.
    By possibly permuting the input variables of $p$, we can assume without loss of generality that $x$ is non-increasing.
    This proves the result.
\end{proof}

\noindent
We now utilize this result to prove the productization result for homogeneous real stable polynomials, and in particular for polynomials in $\HStab_n(\alpha)$ and in $\NHStab_n^d(\alpha)$.

\begin{theorem} \label{thm:hstab-prod}
    Fix $n,d \in \N$, $u,\alpha \in \R_+^n$, and $p \in \R_+^d[x_1,\ldots,x_n]$ such that $p(\mathbf{1}) = 1$ and $\nabla p(\mathbf{1}) = \alpha$. There exists a $d \times n$ matrix $A$ with entries in $\R_+$ such that the rows sums of $A$ are all equal to 1, the column sums of $A$ are given by $\alpha$, and $p(u) = \prod_{i=1}^d (Au)_i$.
\end{theorem}
\begin{proof}
    We first prove the result in the case that $\alpha$ is rational.
    Denote $\alpha = \left(\frac{k_1}{N}, \ldots, \frac{k_n}{N}\right)$ for some $k_1,\ldots,k_n \in \Z_+$ and $N \in \N$.
    Considering variables $y_{1,1},\ldots,y_{1,k_1},y_{2,1},\ldots,y_{n,k_n}$, we define
    \[
        q(y) := p\left(\frac{y_{1,1} + \cdots + y_{1,k_1}}{k_1}, \ldots, \frac{y_{n,1} + \cdots + y_{n,k_n}}{k_n}\right)^N,
    \]
    so that $q \in \R_+^{dN}[y_{1,1},\ldots,y_{n,k_n}]$.
    Since $q(\mathbf{1}) = 1$ and
    \[
        \partial_{y_{i,j}}q(\mathbf{1}) = \left[\frac{N}{k_i} (\partial_{x_i} p) p^{N-1}\right](\mathbf{1}) = 1,
    \]
    we in fact have $q \in \HStab_{dN}(\mathbf{1})$.
    Letting $v \in \R_+^{dN}$ be such that $v_{i,j} = u_i$ for all $i,j$, note that the roots of $q(\mathbf{1}t - v)$ will consist of $N$ copies of the $d$ roots of $p(\mathbf{1}t - u)$.
    So by Theorem \ref{thm:roots-majorization}, there exists $D \in \Mat_{dN}(\mathbf{1})$ such that
    \[
        Dv = (\lambda_1(u), \ldots, \lambda_1(u), \lambda_2(u), \ldots, \lambda_2(u), \ldots, \lambda_d(u), \ldots, \lambda_d(u))
    \]
    where the roots are all repeated $N$ times.
    Let $D'$ be the $d \times n$ matrix formed by summing the elements of each $N \times k_i$ block of $D$ and dividing by $N$.
    We then have
    \[
        \prod_{i=1}^d (D'u)_i = \prod_{i=1}^d \lambda_i(u) = (-1)^n p(1 \cdot 0 - u) = p(u).
    \]
    Since the row sums of $D'$ are all equal to 1 and the column sums are given by $\frac{k_i}{N}$, we have that $D'$ is the desired $d \times n$ matrix with row sums 1 and column sums given by $\alpha$, which proves the result for $p$.
    
    We now handle the case of irrational $\alpha$.
    First if $\alpha_k = 0$ for some $k$, then $p$ does not depend on $x_k$ and the result follows by induction.
    So we may assume that $\alpha_k > 0$ for all $k \in [n]$.
    By \cite{nuij1969note}, the set of homogeneous real stable polynomials of degree $d$ in $n$ variables is the closure of its interior with respect to the Euclidean topology on coefficients.
    Define the map $M(q) := \nabla q(\mathbf{1})$ on the space of $q \in \R^d[x_1,\ldots,x_n]$ for which $q(\mathbf{1}) = 1$, and note that this map is linear and surjects onto the affine subspace of $\R^n$ consisting of vectors whose entries sum to $d$.
    %
    %
    %
    Choosing a small neighborhood $U$ about $p$, surjectivity and linearity imply $M(U)$ contains a small open ball about $\alpha$ in the range of $M$.
    We can therefore choose a sequence $p_j \in \R_+^d[x_1,\ldots,x_n]$ such that $p_j(\mathbf{1}) = 1$, $\nabla p_j(\mathbf{1}) = \alpha^j$, and $\alpha^j$ is rational for all $j$, and $\alpha^j \to \alpha$ and $p_j \to p$.
    The previous arguments then imply there exists a $d \times n$ matrix $A_j$ with row sums all equal to 1 and columns sums given by $\alpha^j$, such that $p_j(u) = \prod_{i=1}^d (A_j u)_i$ for all $j$.
    By compactness of the set of all $d \times n$ matrices with non-negative entries and row sums all equal to 1, we can assume that $A_j$ is a convergent subsequence with limit $A$.
    Therefore $A$ is a $d \times n$ matrix with non-negative entries such that the row sums of $A$ are all equal to 1, the column sums are given by $\alpha$, and $p(x) = \prod_{i=1}^d (Au)_i$. This completes the proof.
    
\end{proof}

\begin{corollary} \label{cor:hstab-prod}
    Fix $n \in \N$, $u,\alpha \in \R_+^n$, and $p \in \HStab_n(\alpha)$. There exists $f \in \Prod_n(\alpha)$ such that $p(u) = f(u)$.
\end{corollary}

\begin{corollary} \label{cor:nhstab-prod}
    Fix $n,d \in \N$, $u,\alpha \in \R_+^n$, and $p \in \NHStab_n^d(\alpha)$. There exists $f \in \NHProd_n^d(\alpha)$ such that $p(u) = f(u)$.
\end{corollary}
\begin{proof}
    Let $q(x) = x_{n+1}^d \cdot p\left(\frac{x_1}{x_{n+1}}, \ldots, \frac{x_n}{x_{n+1}}\right)$ be the homogenization of $p$, and define $\beta := \nabla q(\mathbf{1})$. So $q \in \R_+^d[x_1,\ldots,x_{n+1}]$ such that $q(\mathbf{1}) = 1$ and $\nabla q(\mathbf{1}) = \beta = (\alpha_1,\ldots,\alpha_n,d-\|\alpha\|_1)$. Define $u_{n+1} := 1$, apply Theorem \ref{thm:hstab-prod} to $q$ and $u$, and dehomogenize to obtain the desired result. 
\end{proof}

\noindent
The perturbation argument at the end of the proof of Theorem \ref{thm:hstab-prod} can also be replaced by a different argument which uses the fact that $r \mapsto \left.\nabla p(r \cdot x)\right|_{x=\mathbf{1}}$ maps the strict positive orthant to the relative interior of the Newton polytope of $p$.
With this, we can choose $p_j \in \R_+^d[x_1,\ldots,x_n]$ such that $p(\mathbf{1}) = 1$ and $\nabla p(\mathbf{1}) = \alpha^j$ for $\alpha^j$ rational by choosing particular values of $r^j$ which limit to $\mathbf{1}$.
As a note, both arguments work for both log-concave and strongly log-concave polynomials (topological properties of the set of strongly log-concave polynomials follow from results of \cite{branden2019lorentzian}).

We now prove the capacity bound for real stable polynomials.
Recall the definition of $L_n(\alpha)$ given in Section \ref{sec:cap-bound-prod}.

\begin{corollary} \label{cor:hstab-cap-bound}
    For $p \in \HStab_n(\alpha)$, we have
    \[
        \cpc_\mathbf{1}(p) \geq L_n(\alpha) \geq \left(1 - \frac{\|1-\alpha\|_1}{2}\right)^n.
    \]
\end{corollary}
\begin{proof}
    The second inequality is given by Theorem \ref{thm:prod-bound}, so we just need to prove the first inequality.
    For any $x \in \R_+^n$, let $f \in \Prod_n(\alpha)$ be such that $p(x) = f(x)$ according to Corollary \ref{cor:hstab-prod}.
    With this, we have
    \[
        \cpc_\mathbf{1}(p) = \inf_{x > 0} \frac{p(x)}{x^\mathbf{1}} \geq \inf_{x > 0} \min_{f \in \Prod_n(\alpha)} \frac{f(x)}{x^\mathbf{1}} = L_n(\alpha).
    \]
\end{proof}

\noindent
Note that to get this lower bound on capacity, we actually only needed a lower bound for the productization. That is, we only used the fact that for any $p \in \HStab_n(\alpha)$ and any $x \in \R_+^n$, there is some $f \in \Prod_n(\alpha)$ such that $p(x) \geq f(x)$. Of course, having equality in the productization is a nice fact on its own.

Finally, we apply the productization result to obtain the desired capacity bound for non-homogeneous real stable polynomials. Recall the definition of $L_n^{\NHProd}(\alpha;\kappa)$ given in Section \ref{sec:imp-bound-prod}.

\begin{corollary} \label{cor:imp-cap-bound}
    For $p \in \NHStab_n^d(\alpha)$ and $\kappa \in \Z_+^n$ such that $\|\alpha-\kappa\|_1 \leq 1-\epsilon$, we have
    \[
        \cpc_\kappa(p) \geq L_n^{\NHProd}(\alpha;\kappa) \geq \epsilon^n.
    \]
\end{corollary}
\begin{proof}
    The second inequality is given by Theorem \ref{thm:improved_cap_bound}, so we just need to prove the first inequality. For any $x \in \R_+^n$, let $f \in \NHProd_n^d(\alpha)$ be such that $p(x) = f(x)$ according to Corollary \ref{cor:nhstab-prod}. With this, we have
    \[
        \cpc_\kappa(p) = \inf_{x > 0} \frac{p(x)}{x^\kappa} \geq \inf_{x > 0} \min_{d \in \N} \min_{f \in \NHProd_n^d(\alpha)} \frac{f(x)}{x^\kappa} = L_n^{\NHProd}(\alpha;\kappa).
    \]
\end{proof}

\section{Application to Metric TSP} \label{sec:metric_TSP}

In a recent paper \cite{karlin2020slightly}, Karlin, Klein and Oveis Gharan give an improved approximation factor for metric TSP.
Their proof relies on bounds similar to that of Theorem \ref{thm:main} and Theorem \ref{thm:main_nh}.
In this section, we discuss how our bounds relate to their bounds.

First we need to set up a bit of their notation.
Let $\mu$ be a probability distribution on $\{0,1\}^m$, and let the corresponding probability generating function be given by
\[
    p_\mu(z) := \sum_{S \subseteq [m]} \mathbb{P}(1_S) z^S.
\]
Such a distribution $\mu$ is called \emph{strongly Rayleigh (SR)} when the polynomial $p$ is real stable.
Let $X$ be a random variable distributed according to $\mu$.
We then want to investigate random variables $A_1, \ldots, A_n$ which are defined via sets $S_1 \sqcup \cdots \sqcup S_n = [m]$ by
\[
    A_i := \sum_{s \in S_i} X_s.
\]
That is, $A_i$ is the random variable given by summing the entries of $X$ corresponding to $S_i$.
Our Theorem \ref{thm:main} then naively implies the following bound.
Note that for large $d$, this is much weaker than what we achieve below in Corollary \ref{cor:imp_metric_TSP}.

\begin{corollary} \label{cor:metric_TSP}
    Let $\mu$ be a strongly Rayleigh distribution on $\{0,1\}^m$, and let $A_1,\ldots,A_n$ be random variables corresponding to sets $S_1,\ldots,S_n$ as described above.
    Define $\beta_i := \mathbb{E}[A_i]$, and let $d$ be the size of the largest set that $\mu$ assigns a non-zero probability.
    If $\|\beta-\mathbf{1}\|_1 < 1 - \epsilon$, then
    \[
        \mathbb{P}[\forall i : A_i = 1] > e^{-n} \epsilon^d.
    \]
\end{corollary}
\begin{proof}
    To prove this, we translate the above discussion into the language of polynomials.
    Given $\mu$ and $\mathcal{S} = (S_1,\ldots,S_n)$, define
    \[
        p_{\mu,\mathcal{S}}(x_1,\ldots,x_n) := \left.p_\mu(z)\right|_{z_i = x_j \text{ for } i \in S_j}.
    \]
    So $p_{\mu,\mathcal{S}}$ is a polynomial in $n$ variables of degree at most $m$, and the coefficient of $x^\kappa$ is the probability that $(A_1, \ldots, A_n) = \kappa$.
    In particular, we want to bound the coefficient of $x^{\mathbf{1}}$ in $p_{\mu,\mathcal{S}}$.
    Further, we also define $P_{\mu,\mathcal{S}}(x_1,\ldots,x_n,x_{n+1})$ as the homogenization of $p_{\mu,\mathcal{S}}$, and $d \leq m$ is its degree.
    Since setting variables equal and homogenization are operations which preserve real stability when the coefficients are non-negative, we have that $P_{\mu,\mathcal{S}}$ is real stable when $\mu$ is SR.
    We finally define
    \[
        Q_{\mu,\mathcal{S}}(x_1,\ldots,x_d) := P_{\mu,\mathcal{S}}\left(x_1,\ldots,x_n, \frac{x_{n+1}+\cdots+x_d}{d-n}\right)
    \]
    which is then also a real stable homogeneous polynomial, of degree $d$ in $d$ variables.
    
    We now apply our bound to the polynomial $Q_{\mu,\mathcal{S}}$.
    Since $\beta := \nabla p_{\mu,\mathcal{S}}(\mathbf{1})$, we have
    \[
        \alpha := \nabla Q_{\mu,\mathcal{S}}(\mathbf{1}) = \left(\beta_1,\ldots,\beta_n,\frac{d-\|\beta\|_1}{d-n}\right).
    \]
    Note further that $Q_{\mu,\mathcal{S}}(\mathbf{1}) = 1$.
    Now, $\|\beta-\mathbf{1}\|_1 < 1 - \epsilon$ then implies
    \[
    \begin{split}
        \|\alpha-\mathbf{1}\|_1 &= \|\beta-\mathbf{1}\|_1 + (d-n) \left|\frac{d-\|\beta\|_1}{d-n} - 1\right| \\
            &= \|\beta-\mathbf{1}\|_1 + \left|\sum_{i=1}^n 1 - \beta_i\right| \\
            &< 2(1 - \epsilon).
    \end{split}
    \]
    Applying our theorem then gives
    \[
        \inf_{x_1,\ldots,x_d > 0} \frac{Q_{\mu,\mathcal{S}}}{x_1 \cdots x_d} \geq \left(1 - \frac{\|1-\alpha\|_1}{2}\right)^d > \epsilon^d.
    \]
    
    Additionally, it is easy to see that the $x^{\mathbf{1}}$ coefficients of $Q_{\mu,\mathcal{S}}$ and $p_{\mu,\mathcal{S}}$ are related via
    \[
        \langle x^{\mathbf{1}} \rangle p_{\mu,\mathcal{S}} = \frac{(d-n)^{d-n}}{(d-n)!} \cdot \langle x^{\mathbf{1}} \rangle Q_{\mu,\mathcal{S}}.
    \]
    We finally apply Gurvits' original coefficient bound (see Theorem 1.4 of \cite{karlin2020slightly}) to get
    \[
        \langle x^{\mathbf{1}} \rangle Q_{\mu,\mathcal{S}} \geq \frac{d!}{d^d} \cdot \inf_{x_1,\ldots,x_d > 0} \frac{Q_{\mu,\mathcal{S}}}{x_1 \cdots x_d} > \frac{d!}{d^d} \cdot \epsilon^d.
    \]
    Combining everything then gives
    \[
        \mathbb{P}[\forall i : A_i = 1] = \langle x^{\mathbf{1}} \rangle p_{\mu,\mathcal{S}} > \frac{(d-n)^{d-n}}{(d-n)!} \cdot \frac{d!}{d^d} \cdot \epsilon^d \geq e^{-n} \epsilon^d,
    \]
    as desired.
\end{proof}

\noindent
Our Theorem \ref{thm:main_nh} then implies the following bound. This bound is a major improvement over that of Corollary \ref{cor:metric_TSP} above.

\begin{corollary} \label{cor:imp_metric_TSP}
    Let $\mu$ be a strongly Rayleigh distribution on $\{0,1\}^m$, and let $A_1,\ldots,A_n$ be random variables corresponding to sets $S_1,\ldots,S_n$ as described above.
    Let $\beta_i := \mathbb{E}[A_i]$ and fix $\kappa \in \Z_+^n$. If $\|\beta-\kappa\|_1 \leq 1 - \epsilon$, then
    \[
        \mathbb{P}[\forall i : A_i = \kappa_i] \geq e^{-\|\kappa\|_1} \frac{\kappa^\kappa}{\kappa!} \epsilon^n \geq \epsilon^n \prod_{\kappa_i > 0} \frac{1}{e\sqrt{\kappa_i}}.
    \]
\end{corollary}
\begin{proof}
    As in the proof of Corollary \ref{cor:metric_TSP}, given $\mu$ and $\mathcal{S} = (S_1.\ldots,S_n)$ we define
    \[
        p_{\mu,\mathcal{S}}(x_1,\ldots,x_n) := \left.p_\mu(z)\right|_{z_i = x_j \text{ for } i \in S_j}.
    \]
    So $p_{\mu,\mathcal{S}} \in \NHStab_n^m(\beta)$ and the coefficient of $x^\kappa$ is the probability that $(A_1,\ldots,A_n) = \kappa$. We now want to bound the coefficient of $x^\kappa$ in $p_{\mu,\mathcal{S}}$. We first apply Gurvits' coefficient bound (see Corollary 3.6 of \cite{gurvits2009multivariate}) to get
    \[
        \langle x^\kappa \rangle p_{\mu,\mathcal{S}} \geq e^{-\|\kappa\|_1} \frac{\kappa^\kappa}{\kappa!} \cpc_\kappa(p_{\mu,\mathcal{S}}).
    \]
    By applying the bound of Corollary \ref{cor:imp-cap-bound}, we then obtain
    \[
        \langle x^\kappa \rangle p_{\mu,\mathcal{S}} \geq e^{-\|\kappa\|_1} \frac{\kappa^\kappa}{\kappa!} \epsilon^n.
    \]
    This proves the first inequality, and the second inequality follows from Stirling's approximation.
\end{proof}

\begin{corollary} \label{cor:imp_metric_TSP_allones}
    Let $\mu$ be a strongly Rayleigh distribution on $\{0,1\}^m$, and let $A_1,\ldots,A_n$ be random variables corresponding to sets $S_1,\ldots,S_n$ as described above.
    Let $\beta_i := \mathbb{E}[A_i]$. If $\|\beta-\mathbf{1}\|_1 \leq 1 - \epsilon$, then
    \[
        \mathbb{P}[\forall i : A_i = 1] \geq \left(\frac{\epsilon}{e}\right)^n.
    \]
\end{corollary}

\noindent
Our bound in Corollary \ref{cor:metric_TSP} compares favorably to that of \cite{karlin2020slightly} when $d$ is of order at most $2^n$.
However, the main problem with that bound is its dependence on $d$, which could be as large as $m$.
In \cite{karlin2020slightly}, the authors were able to achieve a bound that was independent of $m$ (and $d$), and this was important to their applications.

Corollary \ref{cor:imp_metric_TSP} and Corollary \ref{cor:imp_metric_TSP_allones} utilize more intricate arguments than what was used to prove Corollary \ref{cor:metric_TSP}; see Section \ref{sec:imp-bound-prod}.
The benefit is a vastly superior bound for large $d$, which is also independent of $m$ and $d$ as required by the arguments of \cite{karlin2020slightly}.
Utilizing these improved bounds in the context of \cite{karlin2020slightly} should lead to a significant further improvement of approximation constant for metric TSP.

\section{Algorithms} \label{sec:algos}

\subsection{Scaling for real stable polynomials}






In this section, we describe the scaling algorithm referenced in Theorem \ref{thm:main-scaling-algo} and demonstrate its rate of convergence
To do this, we first introduce some notation.
We define the vector $x^{(0)} := \bm{1}$ as the starting point for algorithm, and for a given fixed homogeneous polynomial $p$ of degree $n$ in $n$ variables we define
\[
    x_i^{(t+1)} := \frac{c^{(t)}}{\gamma_i^{(t)}} \cdot x_i^{(t)} \quad \text{where} \quad \gamma_i^{(t)} := \frac{x_i^{(t)} \partial_{x_i}p(x^{(t)})}{p(x^{(t)})} \quad \text{and} \quad c^{(t)} := \left(\gamma_1^{(t)} \cdots \gamma_n^{(t)}\right)^{\frac{1}{n}}
\]
for all $t \in \N$.
Note that since $x^{(0)} := \bm{1}$, we have that $\prod_i x_i^{(t)} = 1$ for all $t$ by induction.

The algorithm for scaling is then given precisely by iteratively constructing the vectors $x_i^{(t)}$ as defined above.
We prove the claimed rate of convergence for this algorithm in Proposition \ref{prop:section-scaling-algo} below.
First though, we need a lemma.

\begin{lemma}
    There exists a positive constant $C$ such that for large enough $n$ and all $\gamma \in \R_+^n$ for which $\sum_i \gamma_i = n$, we have
    \[
        \log(\gamma_1 \cdots \gamma_n) \leq \max\left\{-C, -\frac{\|1-\gamma\|_2^2}{6}\right\}.
    \]
    When $\|1 - \gamma\|_1 < 2$, we further have
    \[
        \log(\gamma_1 \cdots \gamma_n) \leq -\frac{\|1-\gamma\|_2^2}{6}.
    \]
\end{lemma}
\begin{proof}
    We split into two cases.
    In the first case, we assume that $\gamma_i \in (0,2)$ for all $i$.
    (Note that $\|1-\gamma\|_1 < 2$ implies that we are in this case since $\sum_i \gamma_i = n$.)
    So we can use the Taylor expansion of $\log(1-x)$ to get
    \[
        \log(\gamma_1 \cdots \gamma_n) = - \sum_{k=1}^\infty \frac{1}{k} \sum_{i=1}^n (1-\gamma_i)^k = -\sum_{i=1}^n (1-\gamma_i) - \sum_{k=2}^\infty \frac{1}{k} \sum_{i=1}^n (1-\gamma_i)^k = -\sum_{k=2}^\infty \frac{1}{k} \sum_{i=1}^n (1-\gamma_i)^k,
    \]
    since $\sum_i \gamma_i = n$.
    Since $\gamma_i \in (0,2)$, we also have that $|1-\gamma_i|^k \geq |1-\gamma_i|^{k+1}$.
    This implies
    \[
    \begin{split}
        -\sum_{k=2}^\infty \frac{1}{k} \sum_{i=1}^n (1-\gamma_i)^k &\leq -\sum_{k=2}^\infty \frac{(-1)^k}{k} \sum_{i=1}^n |1-\gamma_i|^k \\
            &= -\sum_{i=1}^n \frac{|1-\gamma_i|^2}{6} - \sum_{i=1}^n \frac{|1-\gamma_i|^2-|1-\gamma_i|^3}{3} - \sum_{\substack{k = 4 \\ k,~\text{even}}}^\infty \sum_{i=1}^n \left[\frac{|1-\gamma_i|^k}{k} - \frac{|1-\gamma_i|^{k+1}}{k+1}\right] \\
            &\leq -\sum_{i=1}^n \frac{|1-\gamma_i|^2}{6} - 0 - 0.
    \end{split}
    \]
    This proves the result in the first case.
    
    In the second case, we assume that $\gamma_i \geq 2$ for some $i$, and without loss of generality we can assume $i = n$.
    By assumption, we have that $\sum_{i=1}^{n-1} \gamma_i = n - \gamma_n$, and to follows from concavity of $\log$ that $\log(\gamma_i \cdots \gamma_{n-1})$ is maximized over this domain when $\gamma_1 = \gamma_2 = \cdots = \gamma_{n-1} = \frac{n - \gamma_n}{n-1}$.
    With this we have
    \[
        \log(\gamma_1 \cdots \gamma_n) \leq \log\left[\left(1 - \frac{\gamma_n-1}{n-1}\right)^{n-1}\right] + \log \gamma_n \approx 1 - \gamma_n + \log \gamma_n,
    \]
    for large enough $n$.
    The right-hand side is decreasing in $\gamma_n$ for $\gamma_n \geq 2$, and so
    \[
        1 - \gamma_n + \log \gamma_n \leq -1 + \log 2 = -C.
    \]
    This completes the proof.
\end{proof}

\noindent
With this, we now prove the rate of convergence of the above iterative procedure.

\begin{proposition} \label{prop:section-scaling-algo}
    Let $p$ be a real stable homogeneous polynomial of degree $n$ in $n$ variables such that $p(\bm{1}) = 1$ and $\cpc_{\bm{1}}(p) > 0$.
    The above iterative procedure gives an algorithm for approximating, monotonically from above, the value of $\cpc_{\bm{1}}(p)$.
    %
    The rate of convergence is $O(1/t)$ where $t$ is the number of iterations and $O$ depends on $n$ and the value of $\cpc_{\bm{1}}(p)$.
\end{proposition}
\begin{proof}
    %
    Applying Proposition \ref{prop:other-bounds} gives
    \[
        \frac{p(x^{(t+1)})}{p(x^{(t)})} \leq \left(\frac{1}{n} \sum_{i=1}^n \gamma_i^{(t)} \cdot \frac{x_i^{(t+1)}}{x_i^{(t)}}\right)^n = \left(\frac{1}{n} \sum_{i=1}^n \gamma_i^{(t)} \cdot \frac{c^{(t)}}{\gamma_i^{(t)}}\right)^n = \gamma_1^{(t)} \cdots \gamma_n^{(t)}.
    \]
    By definition of $\gamma_i^{(t)}$, Euler's identity for homogeneous functions gives
    \[
        \sum_{i=1}^n \gamma_i^{(t)} = \sum_{i=1}^n \frac{x_i^{(t)} \partial_{x_i}p(x^{(t)})}{p(x^{(t)})} = \frac{n \cdot p(x^{(t)})}{p(x^{(t)})} = n.
    \]
    %
    By the previous lemma, this implies
    \[
        \log\left(p(x^{(t+1)})\right) - \log\left(p(x^{(t)})\right) \leq \log\left(\gamma_1^{(t)} \cdots \gamma_n^{(t)}\right) \leq \max\left\{-C, -\frac{\|1-\gamma^{(t)}\|_2^2}{6}\right\}
    \]
    for $n$ large enough where $C$ is the positive constant from the previous lemma.
    Defining $\epsilon_t := \log p(x^{(t)}) - \log \cpc_{\bm{1}}(p)$, this implies
    \[
        \epsilon_{t+1} \leq \epsilon_t + \max\left\{-C, -\frac{\|1-\gamma^{(t)}\|_2^2}{6}\right\}.
    \]
    We now set out to describe the rate at which the error $\epsilon_t$ goes to 0.
    
    First suppose that $\|1 - \gamma^{(0)}\|_1 \geq 2$.
    Then $\|1 - \gamma^{(t)}\|_2^2 \geq \frac{1}{n} \|1 - \gamma^{(t)}\|_1^2 \geq \frac{4}{n}$, which implies
    \[
        \epsilon_{t+1} \leq \epsilon_t + \max\left\{-C, -\frac{\|1-\gamma^{(t)}\|_2^2}{6}\right\} \leq \epsilon_t + \max\left\{-C, -\frac{2}{3n}\right\} \leq \epsilon_t - \frac{2}{3n}.
    \]
    Since $\log\cpc_{\bm{1}}$ is a finite constant, this implies that either the error is 0 or $\|1 - \gamma^{(t)}\|_1 < 2$ after a number of iterations which is linear in $n$.
    
    Therefore, we from now on assume that $\|1 - \gamma^{(0)}\|_1 < 2$.
    This allows us to use our main bound from Theorem \ref{thm:main}, which implies
    \[
        \epsilon_t = \log p(x^{(t)}) - \log \cpc_{\bm{1}}(p) \leq 0 - n \cdot \log \left(1 - \frac{1}{2} \|1-\gamma^{(t)}\|_1\right).
    \]
    The inequality $\log p(x^{(t)}) \leq 0$ follows inductively from the fact that our starting point above was $x = \bm{1}$ (for which this inequality holds) and at every step $\epsilon_t$ is decreasing.
    Using the Taylor expansion for $\log$ and the inequality between the 1-norm and the 2-norm, we then have
    \[
        \epsilon_t \leq -n \cdot \log \left(1 - \frac{1}{2} \|1-\gamma^{(t)}\|_1\right) \leq -n \cdot -\frac{1}{2} \|1-\gamma^{(t)}\|_1 \leq \frac{n\sqrt{n}}{2} \|1-\gamma^{(t)}\|_2.
    \]
    Rearranging this gives
    \[
        -\frac{\|1-\gamma^{(t)}\|_2^2}{6} \leq -\frac{2}{3n^3} \cdot \epsilon_t^2.
    \]
    Using the previous lemma in the case that $\|\bm{1}-\gamma\|_1 < 2$, we then have
    \[
        \epsilon_{t+1} \leq \epsilon_t - \frac{\|\bm{1}-\gamma^{(t)}\|_2^2}{6} \leq \epsilon_t - \frac{2}{3n^3} \cdot \epsilon_t^2.
    \]
    Denoting $\delta_t := \frac{2}{3n^3} \cdot \epsilon_t$, this implies
    \[
        \delta_{t+1} \leq \delta_t - \delta_t^2.
    \]
    We claim that this implies
    \[
        \delta_t \leq \frac{1}{\delta_0^{-1} + t}.
    \]
    To see this, first note that it is trivial for the base case $t = 0$.
    Then by induction, we have
    \[
        \delta_{t+1} \leq \delta_t - \delta_t^2 \leq \frac{1}{\delta_0^{-1} + t} - \frac{1}{(\delta_0^{-1} + t)^2} = \frac{\delta_0^{-1} + t - 1}{\delta_0^{-1} + t} \cdot \frac{1}{\delta_0^{-1} + t} \leq \frac{\delta_0^{-1} + t}{\delta_0^{-1} + t + 1} \cdot \frac{1}{\delta_0^{-1} + t} = \frac{1}{\delta_0^{-1} + t + 1}.
    \]
    This finally implies
    \[
        \epsilon_t \leq \frac{3n^3}{2} \cdot \frac{1}{\frac{3n^3}{2\epsilon_0} + t} = \frac{\epsilon_0}{1 + \frac{2\epsilon_0}{3n^3} \cdot t},
    \]
    which is what was to be proven.
    
\end{proof}


\begin{remark} \label{rem:bregman}
    There is another iterative algorithm for scaling a polynomial $p_0$ to have marginals equal to $\bm{1}$, via Bregman projections in terms of standard Kullback-Leibler divergence between the vectors of coefficients of polynomials (see \cite{bregman1967relaxation}).
    This algorithm is given as follows.
    First, project $p_0$ onto the space of polynomial for which $p(\bm{1}) = 1$ and $\partial_{x_1} p(\bm{1}) = 1$ to get $p_1$.
    Next, project $p_1$ onto the space of polynomial for which $p(\bm{1}) = 1$ and $\partial_{x_2} p(\bm{1}) = 1$ to get $p_2$.
    Continue this process for all partial derivatives $\partial_{x_1},\partial_{x_2},\partial_{x_3},\ldots$, and once all variables have been used, start the process over again at $\partial_{x_1}$.
    If $\cpc_{\bm{1}}(p_0) > 0$, then this process will converge.
    This iterative algorithm is not the same as ours given above.
    In fact, this Bregman projections approach works for general polynomials with non-negative coefficients, while our approach requires restrictions on the polynomial (log-concavity, at least).
    The benefit of our algorithm is that it is simpler to implement.
\end{remark}

\subsection{An algorithm for computing $L_n(\alpha)$} \label{sec:algo-min-capacity}

In this section, we give an algorithm for computing the minimum capacity value $L_n(\alpha)$ for any fixed $\alpha > 0$.
Recall
\[
    L_n(\alpha) = \min_{p \in \Prod_n(\alpha)} \cpc_{\mathbf{1}}(p) = \min_{p \in \HStab_n(\alpha)} \cpc_{\mathbf{1}}(p),
\]
where the second equality follows from the results of the previous section.
To compute this minimum, first note that
\[
    f(M) := \log\cpc_{\mathbf{1}}\left(\prod_{i=1}^n (Mx)_i\right) = \inf_{x > 0} \left(\sum_{i=1}^n \log(Mx)_i - \alpha_i \log x_i\right)
\]
is concave as a function of $M \in \Mat_n(\alpha)$.
This follows from the fact that $\sum_{i=1}^n \log(Mx)_i - \alpha_i \log x_i$ is a concave function in $M$ for all $x$, and concavity is preserved under taking $\inf$.
Therefore to compute
\[
    L_n(\alpha) = \min_{p \in \Prod_n(\alpha)} \cpc_{\mathbf{1}}(p) = \min_{M \in \Mat_n(\alpha)} e^{f(M)},
\]
we just need to minimize $f(M)$ over the extreme points of $\Mat_n(\alpha)$.
The support (non-zero entries) of the extreme points of $\Mat_n(\alpha)$ then correspond to bipartite forests on $2n$ vertices.
(To see this, note that if the support of $M$ contains a cycle, then one can perturb the corresponding entries by $\pm \epsilon$ with alternating sign to show that $M$ is not extreme.)
Now let $M$ be an extreme point of $\Mat_n(\alpha)$ with support $F$ which is a bipartite forest on $2n$ vertices.
Then there is some row or column of $M$ which contains exactly one non-zero element, corresponding to an edge connected to a leaf of $F$.
The appropriate row or column sum then forces a specific value for this entry of $M$.
Remove that edge from $F$, and remove the corresponding row or column from $M$.
Since $F$ is still a forest after this, we can recursively apply the above argument.
This implies the entries of $M$ are actually determined by $F$.
So for every bipartite forest $F$ on $2n$ vertices, there is at most one $M$ with support $F$.
The above argument also describes the algorithm for constructing the matrix $M$ from $F$.
(If at any point a row or column sum is violated, it means there is no such $M$ with support $F$.)
These observations then yield an algorithm for computing $L_n(\alpha)$, given as follows.

\begin{enumerate}
    \item Iterate over all bipartite forests $F$ on $2n$ vertices.
    \item Construct the matrix $M \in \Mat_n(\alpha)$ associated to $F$, or skip this $F$ if no such $M$ exists.
    \item Compute $f(M)$, keeping track of the minimum value.
\end{enumerate}

\noindent
This algorithm has running time on the order the number of spanning forests of the complete bipartite graph on $2n$ vertices, which is at least $n^{2(n-1)}$.

\section{More General Polynomial Classes} \label{sec:other-bounds}

\subsection{Bounds for other polynomial classes}

We have not yet been able to prove the same capacity bound for log-concave or strongly log-concave polynomials.
In this section, we discuss a number of results and observations which suggest that such a bound should be possible.
The main thing we are missing is a productization result for strongly log-concave polynomials.
For real stable polynomials, we were able to explicitly construct the matrices which gave rise to the productization.
The first lemma here shows that a productization result already follows from a bound by the min and max product polynomials.

\begin{lemma} \label{lem:squeeze}
    Fix $n \in \N$, $x,\alpha \in \R_+^n$, and $p \in \R_+^n[x_1,\ldots,x_n]$ such that $p(\mathbf{1}) = 1$ and $\nabla p(\mathbf{1}) = \alpha$. Suppose further that
    \[
        \max_{f \in \Prod_n(\alpha)} f(x) \geq p(x) \geq \min_{f \in \Prod_n(\alpha)} f(x).
    \]
    Then there exists $f \in \Prod_n(\alpha)$ such that $p(x) = f(x)$.
\end{lemma}
\begin{proof}
    Define the map $P: \Mat_n(\alpha) \to \R_+$ via
    \[
        P(A) := \prod_{i=1}^n (Ax)_i.
    \]
    Since $\Mat_n(\alpha)$ is a closed convex polytope, its image under $P$ is a closed interval. The result follows.
\end{proof}

\noindent
Further, we actually only need the productization lower bound to obtain capacity lower bounds.
That is, to prove a capacity bound for (strongly) log-concave polynomials, we just need to prove for any $x$ and $\alpha$ that
\[
    \min_{p \in \LC_n(\alpha)} p(x) \geq \min_{f \in \Prod_n(\alpha)} f(x).
\]
That said, we now state various relations between these upper and lower bounds for the various classes of polynomials.

\begin{proposition} \label{prop:other-bounds}
    Fix $n \in \N$ and $y \in \R_+^n$. Given $p \in \R_+^n[x_1,\ldots,x_m]$, define $\gamma = \left.\nabla \frac{p(y \cdot x)}{p(y)}\right|_{x=\mathbf{1}}$. For any $x \in \R_+^m$, we have
    \begin{enumerate}
        \item $\displaystyle \frac{p(x)}{p(y)} \leq \frac{1}{n} \sum_{i=1}^m \gamma_i \left(\frac{x_i}{y_i}\right)^n$ for all $p \in \R_+^n[x_1,\ldots,x_m]$.
        \item $\displaystyle \frac{p(x)}{p(y)} \leq \left(\frac{1}{n} \sum_{i=1}^m \gamma_i \cdot \frac{x_i}{y_i}\right)^n$ for all $p \in \R_+^n[x_1,\ldots,x_m]$ which is log-concave in $\R_+^m$.
        \item $\displaystyle \frac{p(x)}{p(y)} \geq \left(\frac{x_1}{y_1}\right)^{\gamma_1} \cdots \left(\frac{x_m}{y_m}\right)^{\gamma_m}$ for all $p \in \R_+^n[x_1,\ldots,x_m]$.
    \end{enumerate}
\end{proposition}
\begin{proof}
    Since $\frac{p(x)}{p(y)} = \frac{p(y \cdot x)}{p(y)}$ and variable scaling preserves the various classes of polynomials, we only need to prove the bounds for $y = \mathbf{1}$ and $p(\mathbf{1}) = 1$.
    
    $(1)$. By AM-GM, we have $x^\mu \leq \sum_{i=1}^m \frac{\mu_i}{n} x_i^n$ for any $\mu \in \Z_+^n$ and any $x \in \R_+^n$. Therefore
    \[
        p(x) = \sum_\mu p_\mu x^\mu \leq \sum_\mu p_\mu \sum_{i=1}^m \frac{\mu_i x_i^n}{n} = \sum_{i=1}^m \frac{x_i^n}{n} \sum_\mu \mu_i p_\mu = \sum_{i=1}^m \frac{x_i^n}{n} \cdot \gamma_i.
    \]
    
    $(2)$. By log-concavity, $p^{\frac{1}{n}}$ is concave in the positive orthant. Further, we have
    \[
        p^{\frac{1}{n}}(\mathbf{1}) = \frac{\sum_{i=1}^m \gamma_i \cdot 1}{n} \qquad \text{and} \qquad \nabla|_{x=\mathbf{1}} (p^{\frac{1}{n}}) = \left(\frac{\gamma_1}{n}, \ldots, \frac{\gamma_m}{n}\right) = \nabla|_{x=\mathbf{1}} \left(\frac{\sum_{i=1}^n \gamma_i x_i}{n}\right).
    \]
    Since $\frac{1}{n} \sum_{i=1}^n \gamma_i x_i$ is linear and $p^{\frac{1}{n}}$ is concave, this immediately implies $p(x) \leq \left(\frac{1}{n}\sum_{i=1}^n \gamma_i x_i\right)^n$.
    
    $(3)$. Proposition \ref{prop:ds-cap} implies $\frac{p(x)}{x^\gamma} \geq 1$, which gives the bound.
\end{proof}

\noindent
An immediate corollary of $(2)$ in the above proposition is that the maximizing polynomial for log-concave polynomials is a product polynomial, stated formally as follows.

\begin{corollary} \label{cor:prod-lc-max}
    For $n \in \N$ and $x,\alpha \in \R_+^n$, we have
    \[
        \max_{p \in \Prod_n(\alpha)} p(x) = \left(\frac{\sum_{i=1}^n \alpha_i x_i}{n}\right)^n = \max_{p \in \LC_n(\alpha)} p(x).
    \]
\end{corollary}

\noindent
Note also that $(1)$ in the above proposition implies there is no such relationship between product polynomials and polynomials $p \in \R_+^n[x_1,\ldots,x_n]$ in general.

Of course what we really care about here is the lower bound, which in general is more difficult.
In particular, it seems unlikely that the minimums will have explicit formulas like the maximums did in Corollary \ref{cor:prod-lc-max}.
One thing we can say in the direction of a lower bound follows from $(3)$ in the above proposition, stated as follows.

\begin{corollary}
    For $n \in \N$, $x \in \R_+^n$, and non-negative integer vector $\alpha \in \Z_+^n$, we have
    \[
        \min_{p \in \Prod_n(\alpha)} p(x) = \prod_{i=1}^n x_i^{\alpha_i} = \min_{p \in \R_+^n[x_1,\ldots,x_n]} p(x).
    \]
\end{corollary}
\begin{proof}
    Follows from the fact that $\prod_{i=1}^n x_i^{\alpha_i}$ is a product polynomial when $\alpha \in \Z_+^n$.
\end{proof}

\subsection{The general minimization problem} \label{sec:general-polynomials}

In general, the minimization problem for general polynomials $p \in \R_+^n[x_1,\ldots,x_m]$ for which $p(\mathbf{1}) = 1$ and $\nabla p(\mathbf{1}) = \alpha$ can be written as the linear program
\[
    \min_{\substack{p_\mu \geq 0 \\ \sum_\mu \mu \cdot p_\mu = \alpha}} \quad \sum_\mu p_\mu t^\mu
\]
where $t > 0$ is fixed.
(Note that homogeneity makes the $p(\mathbf{1}) = 1$ condition equivalent to $\sum_i \alpha_i = n$.)
One thing we can do is characterize the support of the minimizers of the above linear program.

\begin{proposition}
    For $t,\alpha \in \R_+^m$, a support set $S$ is the support of a polynomial $p \in \R_+^n[x_1,\ldots,x_m]$ which minimizes the above linear program if and only if the following hold.
    \begin{enumerate}
        \item $\alpha$ is in the convex hull of $S$.
        \item There exists $\beta \in \R^m$ such that
        \[
            (t_1x_1+\cdots+t_mx_m)^n - (\beta_1x_1+\cdots+\beta_mx_m) (x_1+\cdots+x_m)^{n-1}
        \]
        is supported outside of $S$ and has non-negative coefficients.
    \end{enumerate}
\end{proposition}
\begin{proof}
    $(\implies)$. Property $(1)$ is immediate.
    For property $(2)$, consider the standard dual linear program, along with an equivalent formulation with slack variables:
    \[
        \min_{\substack{y \in \R^m \\ \sum_{i=1}^m \mu_i y_i \leq t^\mu}} \quad \sum_{i=1}^m \alpha_i y_i \qquad \iff \qquad \min_{\substack{y \in \R^m \\ \sum_{i=1}^m \mu_i y_i + s_\mu = t^\mu \\ s_\mu \geq 0}} \quad \sum_{i=1}^m \alpha_i y_i.
    \]
    Now suppose $p^\star$ is an optimum for the primal program with support $S$, and let $y^\star,s^\star$ be an optimum for the dual program with slack variables.
    Strong duality then implies
    \[
        \alpha \cdot y^\star = \sum_\mu p^\star_\mu t^\mu = \sum_\mu p^\star_\mu \left[\mu \cdot y^\star + s^\star_\mu\right] = \alpha \cdot y^\star + \sum_\mu p^\star_\mu s^\star_\mu \implies \sum_\mu p^\star_\mu s^\star_\mu = 0.
    \]
    Therefore $s^\star_\mu = 0$ for $\mu \in S$, which implies $\sum_{i=1}^m \mu_i y^\star_i = t^\mu$ for $\mu \in S$.
    With this, define
    \[
    \begin{split}
        q(x) := \sum_\mu \binom{n}{\mu} x^\mu \left[t^\mu - \sum_{i=1}^m \mu_i y^\star_i\right] &= (t_1x_1 + \cdots t_mx_m)^n - n \sum_{i=1}^m y^\star_i x_i \sum_\mu \binom{n-1}{\mu-\delta_i} x^{\mu-\delta_i} \\
            &= (t_1x_1 + \cdots t_mx_m)^n - n (y_1^\star x_1 + \cdots y_m^\star x_m) \cdot (x_1 + \cdots x_m)^{n-1}.
    \end{split}
    \]
    Letting $\beta := ny^\star$ completes this direction of the proof.
    
    $(\impliedby)$. Now let $p$ be a polynomial with support $S$ for which $\nabla p(\mathbf{1}) = \alpha$, and let
    \[
        q(x) = (t_1x_1+\cdots+t_mx_m)^n - (\beta_1x_1+\cdots+\beta_mx_m) (x_1+\cdots+x_m)^{n-1}.
    \]
    Using the Bombieri (Fischer-Fock, etc.) inner product, we have
    \[
        0 = \langle p, q \rangle = p(t) - \frac{\beta}{n} \cdot \nabla p(\mathbf{1}) = p(t) - \frac{\alpha \cdot \beta}{n} \implies p(t) = \frac{\alpha \cdot \beta}{n}.
    \]
    Further, for any polynomial $p$ with non-negative coefficients and $\nabla p(\mathbf{1}) = \alpha$ we have
    \[
        0 \leq \langle p, q \rangle = p(t) - \frac{\alpha \cdot \beta}{n} \implies p(t) \geq \frac{\alpha \cdot \beta}{n}.
    \]
    Therefore $p$ minimizes the primal linear program, and this completes this direction of the proof.
\end{proof}

\noindent
For rational gradient, a completely general lower bound in terms of product polynomials is unlikely to hold.
To see this, note the following support condition implied by the above results.

\begin{lemma}
    Fix $n \in \N$ and $\alpha \in \R_+^n$, and let $p \in \R_+^n[x_1,\ldots,x_n]$ be such that $p(\mathbf{1}) = 1$ and $\nabla p(\mathbf{1}) = \alpha$. If $\cpc_{\mathbf{1}}(p) \geq L_n(\alpha)$, then either $\mathbf{1} \in \Newt(p)$ or $\|\mathbf{1}-\alpha\|_1 \geq 2$.
\end{lemma}
\begin{proof}
    If $\cpc_{\mathbf{1}}(p) > 0$, then $\mathbf{1} \in \Newt(p)$ by Proposition \ref{prop:cap-Newt}. Otherwise, $0 = \cpc_{\mathbf{1}}(p) \geq L_n(\alpha)$ which implies $\|\mathbf{1}-\alpha\|_1 \geq 2$ by Proposition \ref{prop:hall-cap-norm}.
\end{proof}

\noindent
That is, a general lower bound by product polynomials is contradicted by the existence of a polynomial with non-negative coefficients for which $\nabla p(\mathbf{1})$ is close to $\mathbf{1}$ but $\mathbf{1} \not\in \Newt(p)$.
Such a polynomial likely exists.
On the other hand, the well-known matroidal support conditions of strongly log-concave polynomials imply that this polynomial cannot be in $\SLC_n(\alpha)$, and so a lower bound by product polynomials is still possible in the strongly log-concave case.

\subsection{Capacity upper bounds} \label{sec:cap-upper-bounds}

Although we are mainly interested in lower bounds on capacity in this paper, upper bounds can also be achieved using similar methods.
In this section we use the results of Section \ref{sec:other-bounds} to prove upper bounds on the capacity of various classes of polynomials.
We present these observations for the interested reader, but say nothing further about them.
The first is a tight bound for $p \in \LC_n(\alpha)$ which is also tight for $\Prod_n(\alpha)$.

\begin{proposition}
    For $n \in \N$, $\alpha \in \R_+^n$, and $p \in \LC_n(\alpha)$, we have
    \[
        \cpc_{\mathbf{1}}(p) \leq \prod_{i=1}^n \alpha_i.
    \]
    This bound is tight over all $p \in \Prod_n(\alpha)$, and hence tightness also holds for $p \in \LC_n(\alpha)$, $p \in \SLC_n(\alpha)$, and $p \in \HStab_n(\alpha)$.
\end{proposition}
\begin{proof}
    By $(2)$ of Proposition \ref{prop:other-bounds} and Lemma \ref{lem:linear-cap}, we have
    \[
        \cpc_{\mathbf{1}}(p) = \cpc_{\mathbf{1}}\left(\left(\frac{1}{n} \sum_{i=1}^n \alpha_i x_i\right)^n\right) = \prod_{i=1}^n \alpha_i.
    \]
    Tightness follows from the fact that $\left(\frac{1}{n} \sum_{i=1}^n \alpha_i x_i\right)^n \in \Prod_n(\alpha)$.
\end{proof}

\noindent
Interestingly this bound does not hold for polynomials in general, for which we obtain a different tight bound.

\begin{proposition}
    For $n \in \N$, $\alpha \in \R_+^n$, and $p \in \R_+^n[x_1,\ldots,x_n]$ such that $p(\mathbf{1}) = 1$ and $\nabla p(\mathbf{1}) = \alpha$, we have
    \[
        \cpc_{\mathbf{1}}(p) \leq \left(\prod_{i=1}^n \alpha_i\right)^{\frac{1}{n}},
    \]
    and this bound is tight.
\end{proposition}
\begin{proof}
    By $(1)$ of Proposition \ref{prop:other-bounds}, we have
    \[
        \cpc_{\mathbf{1}}(p) \leq \cpc_{\mathbf{1}}\left(\frac{1}{n} \sum_{i=1}^n \alpha_i x_i^n\right).
    \]
    We then compute
    \[
        0 = \nabla \log \left(\frac{\frac{1}{n} \sum_{i=1}^n \alpha_i e^{nx_i}}{e^{\mathbf{1} \cdot x}}\right) = \nabla \left[\log\left(\frac{1}{n} \sum_{i=1}^n \alpha_i e^{nx_i}\right) - \mathbf{1} \cdot x\right] \iff \frac{\alpha_i e^{nx_i}}{\frac{1}{n} \sum_{i=1}^n \alpha_i e^{nx_i}} = 1 \quad \forall i.
    \]
    This is equivalent to saying that $\alpha_i e^{nx_i} = \alpha_j e^{nx_j}$ for all $i,j$, which occurs when $x_i = \frac{-\log \alpha_i}{n}$ for all $i$. Plugging this in to the objective function gives
    \[
        \log \cpc_{\mathbf{1}}(p) \leq \log\left(\frac{\frac{1}{n} \sum_{i=1}^n \alpha_i \cdot \alpha_i^{-1}}{(\alpha_1 \cdots \alpha_n)^{-\frac{1}{n}}}\right) = \log\left((\alpha_1 \cdots \alpha_n)^{\frac{1}{n}}\right).
    \]
    Exponentiating gives the bound. Tightness is immediate.
\end{proof}

\noindent
Again, since this paper is predominantly about lower bounds on polynomial capacity, we say nothing further about these upper bounds.

\section*{Acknowledgements}

We are very grateful to Rafael Oliveira and Avi Wigderson for motivating discussions on the topics of the paper.

\bibliographystyle{plain}
\bibliography{bibliography}

\end{document}